\documentclass{amsart}
\usepackage{amssymb,euscript,amsmath, mathrsfs}
\usepackage[dvips]{graphicx}
\usepackage[dvips]{color}

\newcounter{ENUM}

\input xy
\xyoption{all}
\CompileMatrices

\def\<{\langle}
\def\>{\rangle}
\def\0{{{\bf 0}}}

\def\OO{{\mathcal O}}

\def\CF{{\mathcal F}}

\def\CK{{\mathcal K}}
\def\CA{{\mathcal A}}

\def\CT{{\mathcal T}}
\def\CG{{\mathcal G}}
\def\CP{{\mathcal P}}

\def\AA{{\mathbb A}}

\def\FF{{\mathbb F}}
\def\GG{{\mathbb G}}

\def\LL{{\mathbb L}}

\def\QQ{{\mathbb Q}}

\def\tpi{{\tilde \pi}}

\def\trho{{\tilde \rho}}

\def\rhobar{{\overline{\rho}}}
\def\rhobar{{\overline{\rho}}}

\def\unif{\varpi}

\newcommand{\Ind}{\operatorname{Ind}}
\newcommand{\cInd}{\operatorname{c-Ind}}
\newcommand{\Res}{\operatorname{Res}}

\newcommand{\St}{\operatorname{St}}

\newcommand{\inv}{\operatorname{inv}}

\newcommand{\sss}{\mbox{\rm \tiny ss}}

\newcommand{\ur}{\mbox{\rm \tiny ur}}

\def\ZZ{{\mathbb Z}}

\def\Hom{\operatorname{Hom}}

\def\Fr{\operatorname{Fr}}

\def\End{\operatorname{End}}
\def\Aut{\operatorname{Aut}}

\def\Gal{\operatorname{Gal}}
\def\Spf{\operatorname{Spf}}
\def\GL{\operatorname{GL}}

\def\Spec{\operatorname{Spec}}

\def\Res{\operatorname{Res}}

\newcommand{\Id}{\operatorname{Id}}

\newcommand{\margh}[1]{}

\newtheorem{thm}{Theorem}[section]
\newtheorem{theorem}[thm]{Theorem}
\newtheorem{prop}[thm]{Proposition}
\newtheorem{proposition}[thm]{Proposition}
\newtheorem{lemma}[thm]{Lemma}
\newtheorem{cor}[thm]{Corollary}
\newtheorem{corollary}[thm]{Corollary}

\newtheorem{conjecture}[thm]{Conjecture}

\theoremstyle{definition}
\newtheorem{defn}[thm]{Definition}
\newtheorem{definition}[thm]{Definition}

\newtheorem{question}[thm]{Question}

\newtheorem{remark}[thm]{Remark}

\numberwithin{equation}{section}

\def\CGbar{{\overline{\mathcal G}}}
\def\CBbar{{\overline{\mathcal B}}}
\def\CTbar{{\overline{\mathcal T}}}
\def\CUbar{{\overline{\mathcal U}}}
\def\CLbar{{\overline{\mathcal L}}}
\def\CPbar{{\overline{\mathcal P}}}
\def\taubar{{\overline{\tau}}}
\def\tZ{{\tilde Z}}
\def\tkappa{{\tilde \kappa}}
\def\tY{{\tilde Y}}
\def\tsigma{{\tilde \sigma}}
\def\ttau{{\tilde \tau}}
\def\tFr{{\tilde \Fr}}

\newcommand\Rep{\operatorname{Rep}}

\begin{document}
\title{Curtis homomorphisms and the integral Bernstein center for $\GL_n$}
\author{David Helm}

\maketitle

\begin{abstract}
We describe two conjectures, one strictly stronger than the other, that give descriptions of the integral Bernstein center for
$\GL_n(F)$ (that is, the center of the category of smooth $W(k)[\GL_n(F)]$-modules, for $F$ a $p$-adic field and $k$ an algebraically closed
field of characteristic $\ell$ different from $p$) in terms of Galois theory.  Moreover, we show that the weak version of the conjecture (for $m \leq n$),
together with the strong version of the conjecture for $m < n$,
implies the strong conjecture for $\GL_n$.  In a companion paper~\cite{converse} we show that the strong conjecture for $n-1$ implies the weak
conjecture for $n$; thus the two papers together give an inductive proof of both conjectures.  The upshot is a description of the Bernstein center in
purely Galois theoretic terms; previous work of the author shows that this description implies the conjectural ``local Langlands correspondence in families''
of~\cite{emerton-helm}.
\end{abstract}

\section{Introduction}

In~\cite{emerton-helm}, Emerton and the author describe a conjectural ``local Langlands correspondence in families'' for the group
$\GL_n(F)$, where $F$ is a $p$-adic field.  More precisely, we show that given a suitable
coefficient ring $A$ (in particular complete and local with residue characteristic $\ell$ different from $p$), 
and a family of Galois representations $\rho: G_F \rightarrow \GL_n(A)$, there is, up to isomorphism, at most one admissible $A[\GL_n(F)]$-module
$\pi(\rho)$ that ``interpolates the local Langlands correspondence across the family $\rho$'' and satisfies certain technical hypotheses.  (We
refer the reader to~\cite{emerton-helm}, Theorem 1.1.1 for the precise result.)  We further conjecture that such a representation $\pi(\rho)$ exists
for any $\rho$.

The paper~\cite{bernstein3} gives an approach to the question of actually constructing $\pi(\rho)$ from $\rho$.  The key new idea
is the introduction of the integral Bernstein center, which is by definition the center of the category of smooth $W(k)[\GL_n(F)]$-modules.
More prosaically, the integral Bernstein center is a ring $Z$ that acts on every smooth $W(k)[\GL_n(F)]$-module, compatibly with every morphism
between such modules, and is the universal such ring.  The structure of $Z$ encodes deep information about ``congruences'' between 
$W(k)[\GL_n(F)]$-modules (for instance, if two irreducible representations of $\GL_n(F)$ in characteristic zero become isomorphic modulo $\ell$,
the action of $Z$ on these two representations will be via scalars that are congruent modulo $\ell$.)

Morally, the problem of showing that $\pi(\rho)$ exists for all $\rho$ amounts to showing- for a sufficiently general notion of
``congruence''- that whenever there is a congruence between two representations
of $G_F$, there is a corresponding congruence on the other side of the local Langlands correspondence.  It is therefore not
surprising that one can rephrase the problem of constructing $\pi(\rho)$ in terms of the structure of $Z$.
Indeed, Theorem 7.4 of~\cite{bernstein3} reduces the question of the existence of $\pi(\rho)$ to a conjectured relationship
between the ring $Z$ and the deformation theory of mod $\ell$ representations of $G_F$ (Conjecture 7.2 of~\cite{bernstein3}.)

The primary goal of this paper, together with its companion paper~\cite{converse}, is to prove a version of this conjecture, and
thus establish the local Langlands correspondence in families.  More precisely, we introduce a collection of finite type $W(k)$-algebras
$R_{\nu}$ that parameterize representations of the Weil group $W_F$ of $F$ with fixed restriction to prime-to-$\ell$ inertia, and whose
completion at a given maximal ideal is a close variant of a universal framed deformation ring.  We then conjecture that there is a map
$Z \rightarrow R_{\nu}$ that is ``compatible with local Langlands'' in a certain technical sense (see Conjecture~\ref{conj:weak} below for
a precise statement and discussion.)  This conjecture, which we will henceforth call the ``Weak Conjecture'', becomes Conjecture 7.2
of~\cite{bernstein3} after one completes $R_{\nu}$ at a maximal ideal, and hence implies both that conjecture and the existence of $\pi(\rho)$.

If a map $Z \rightarrow R_{\nu}$ of the conjectured sort exists it is natural to ask what the image is.  The ``Strong Conjecture'' 
(Conjecture~\ref{conj:strong}
below) gives a description of this image (and in fact gives a description of the direct factors of $Z$ in purely Galois-theoretic terms.)  As the
names suggest, the ``Strong Conjecture'' implies the ``Weak Conjecture.''

The main result of this paper is that if the weak conjecture holds for all $\GL_m(F)$, with $m$ less than or equal to a fixed $n$, and the strong conjecture
holds for $m < n$, then the strong Conjecture
holds as well for the group $\GL_n(F)$.  In the companion paper~\cite{converse}, we will show that the strong conjecture for $\GL_{n-1}(F)$ implies
the weak conjecture for $\GL_n(F)$.  Since the case $n=1$ is easy (it is a consequence of local class field theory), the two papers together will
establish both conjectures for all $n$, and hence the local Langlands correspondence for $\GL_n$ in families.

Our approach relies on three main ingredients.  The first is an input from finite group theory, namely the endomorphism ring of the Gelfand-Graev
representation $\overline{\Gamma}$ of $\GL_n(\FF_q)$.  In section~\ref{sec:finite} we introduce this ring and describe some of its basic properties, following
Bonnaf{\'e}-Kessar~\cite{gelfand-graev}.  A crucial structure on this endomorphism ring is its canonical symmetrizing form, which Bonnaf{\'e}-Keessar describe
in terms of ``Curtis homomorphisms'' arising from Deligne-Lusztig restriction.  In section~\ref{sec:bernstein} we describe the connection between this endomorphism
ring and the ring $Z$.

The second key ingredient is the behavior of the integral Bernstein center $Z$ with respect to parabolic induction; for a Levi $M$ of $G$
there are natural maps $Z \rightarrow Z_M$ compatible, in a certain sense, with parabolic induction from $M$ to $G$.  In section~\ref{sec:bernstein}
we recall results of~\cite{bernstein1} (c.f. Theorems~\ref{thm:bernstein ind} and~\ref{thm:near saturation}, below) that say that in certain key 
cases the images of these maps are ``large'' in a certain sense, and that
the failure of these maps to have image that is ``as large as possible'' is controlled by the endomorphism ring of a Gelfand-Graev representation.

The third key ingredient is the construction of the rings $R_{\nu}$
which occupies sections~\ref{sec:tame}, \ref{sec:deformations}, and \ref{sec:global galois}.  These moduli spaces admit maps between them coming from taking direct sums of representations;
these maps serve a purpose analogous to the ``parabolic induction'' maps from $Z$ to $Z_M$.  The functions on such spaces also admit subalgebras $B_{q,n}$ that play a role
analogous to the subalgebras of $Z$ arising from the endomorphism ring $\overline{E}_{q,n}$ of a Gelfand-Graev representation.  The strong conjecture
leads us to expect that in fact $\overline{E}_{q,n}$ and $B_{q,n}$ are isomorphic, but it seems difficult to show this directly (although it is easy to show if
one inverts $\ell$).  Instead, we make use of the symmetrizing form on $\overline{E}_{q,n}$ to show that {\em if} there exists a map from $\overline{E}_{q,n}$ to $B_{q,n}$ then it must be
an isomorphism (c.f. sections~\ref{sec:inertial} and \ref{sec:E-B}.)

Once we have established this, our argument goes as follows.  First we show that the strong conjecture holds after inverting $\ell$; this essentially follows
easily from the classical Bernstein-Deligne theory of the Bernstein center over algebraically closed fields.  We then assume the strong conjecture for
$m < n$, and the weak conjecture for $m \leq n$.  This gives us in particular a map $\overline{E}_{q,n} \rightarrow B_{q,n}$ that is necessarily an isomorphism.
Using this, and considering various parabolic restriction maps from $Z$ to various Levi subgroups, together with the corresponding maps on the rings $R_{\nu}$
of representations of $W_F$, we show, using our ``large image'' results for $Z$, that $Z$ must ``fill out'' the entire ring of invariant functions in $R_{\nu}$,
thus proving the strong conjecture for $\GL_n$.

In the process of carrying out this inductive argument we prove that $\overline{E}_{q,n}$ is isomorphic to $B_{q,n}$ for all $n$.  This is a statement purely in finite group
theory that is of independent interest.  We know of no more direct proof of this isomorphism than the one described here.

Throughout this paper we adopt the following conventions: $F$ is a $p$-adic field with residue field $\FF_q$, $k$ is an algebraically
closed field of characteristic $\ell \neq p$, $\CK$ is the field of fractions of $W(k)$, and $\overline{\CK}$ is an algebraic closure of $\CK$.
Algebraic groups over $F$ with be denoted by uppercase mathcal letters $\CT$, $\CG$, etc.; for any such group the corresponding uppercase letters
$T$, $G$, etc. will denote the groups of $F$-points of $\CT$, $\CG$, and so forth.  In particular there is an implicit dependence of $T$ on
$\CT$.

{\em Acknowledgements}
We are grateful to Jean-Francois Dat, Robert Kurinczuk, Vincent Secherre, David Ben-Zvi, and Richard Taylor for helpful conversations and suggestions, and to Gil Moss for his
comments on an earlier draft of this paper.  We are also deeply indebted to Jack Shotton for noticing a serious error in an earlier version of this paper,
and for bringing to our attention the argument of Proposition 7.10 of~\cite{shotton}, which proved crucial to correcting this error.
This research was
partially supported by NSF grant DMS-1161582 and EPSRC grant EP/M029719/1.

\section{Finite groups} \label{sec:finite}

Before beginning our study of the Bernstein center we develop some finite group theory that will be essential for our approach.
Most of the ideas in this section originally appear in work of Bonnaf{\'e}-Kessar~\cite{gelfand-graev}.

Fix distinct primes $p$ and $\ell$, and a power $q$ of $p$.
Let $\CGbar$ be the group $\GL_n$ over $\FF_q$, and let $\overline{G} = \CGbar(\FF_q)$. 
We will consider the representation theory of
$\overline{G}$ over the Witt ring $W(k)$, where $k$ is an algebraic closure of 
$\FF_{\ell}$.  Let $\CK$ be the field of fractions of $W(k)$, and fix an algebraic
closure $\overline{\CK}$ of $\CK$.

Our principal object of study in this section will be the Gelfand-Graev representation
$\overline{\Gamma}$ of $\overline{G}$, with coefficients in $W(k)$.  
Fix a Borel $\CBbar$ in $\CGbar$, with unipotent radical $\CUbar$,
and let $\overline{B}$, $\overline{U}$ denote the $\FF_q$-points of $\CBbar$ and $\CUbar$
respectively.  Also fix a generic character $\Psi: \overline{U} \rightarrow
W(k)^{\times}$.  Then, by definition, we have $\overline{\Gamma} = \cInd_{\overline{U}}^{\overline{G}} \Psi$,
where $\Psi$ is considered as a $W(k)[\overline{U}]$-module that is free over $W(k)$ of rank one,
with the appropriate action of $\overline{U}$.  The module $\overline{\Gamma}$ is then independent
of the choice of $\Psi$, up to isomorphism.

The objective of this first section is to study the endomorphism ring 
$\End_{W(k)[\overline{G}]}(\overline{\Gamma})$, which we denote by $\overline{E}_{q,n}$.
Our main tool for doing so will be the Deligne-Lusztig
induction and restriction functors of Bonnaf{\'e}-Rouquier~\cite{BR}.  Let $\overline{L}$ be the subgroup
of $\overline{G}$ consisting of the $\overline{\FF}_q$-points of a (not necessarily split)
Levi subgroup $\CLbar$ of $\GL_n$, and choose a parabolic subgroup $\CPbar$ of $\GL_n$ whose Levi subgroup is $\CLbar$.
Let $\Rep_{W(k)}(\overline{G})$ and $\Rep_{W(k)}(\overline{L})$ denote the categories of
$W(k)[\overline{G}]$-modules and $W(k)[\overline{L}]$-modules, respectively.  Then Deligne-Lusztig
induction and restriction are functors:
$$i_{\CLbar \subseteq \CPbar}^{\overline{G}}: {\mathcal D}^b(\Rep_{W(k)}(\overline{L})) \rightarrow
{\mathcal D}^b(\Rep_{W(k)}(\overline{G}))$$
$$r^{\CLbar \subseteq \CPbar}_{\overline{G}}: {\mathcal D}^b(\Rep_{W(k)}(\overline{G})) \rightarrow
{\mathcal D}^b(\Rep_{W(k)}(\overline{L})).$$

We will be concerned exclusively with the case where $\CLbar$ is a maximal torus in $\CGbar$.
In this case the effect of Deligne-Lusztig
restriction on $\overline{\Gamma}$ has been described by Bonnaf{\'e}-Rouquier when $\CLbar$ is a Coxeter torus
and by Dudas~\cite{dudas} in general.

\begin{theorem}[Bonnaf{\'e}-Rouquier, Dudas] \label{thm:dudas}
When $\CLbar$ is the standard maximal torus, there is a natural isomorphism:
$$r^{\CLbar \subseteq \CPbar}_{\overline{G}} \overline{\Gamma} \cong W(k)[\overline{L}][-\ell(w)]$$
in ${\mathcal D}^b(\Rep_{W(k)}(\overline{L}))$, where $w$ is the element of the Weyl group of $\CGbar$ such
that $\CPbar^w$ is the standard Borel, $\ell(w)$ is its length, and $[-\ell(w)]$ denotes a cohomological shift.
\end{theorem}
\begin{proof}
This is the main theorem of~\cite{dudas}.
\end{proof}

An immediate consequence of this result is that, when $\overline{T}$ is the $\overline{\FF}_q$-points of
a torus in $\GL_n$, then an endomorphism of $\overline{\Gamma}$ gives rise, by functoriality of
Deligne-Lusztig restriction, to an endomorphism of $W(k)[\overline{T}]$ (or, equivalently,
an element of $W(k)[\overline{T}]$.)  We thus obtain homorphisms:
$$\Phi_{\overline{T}}: \overline{E}_{q,n} \rightarrow W(k)[\overline{T}]$$
for each torus $\CTbar$ in $\CGbar$.  These are integral versions of the classical ``Curtis homomorphisms''.

Over $\overline{\CK}$, it is not difficult to describe 
the structure of $\overline{\Gamma} \otimes \overline{\CK}$, its endomorphism ring,
and the associated Curtis homomorphisms. Recall
that an irreducible representation $\pi$ of $\overline{G}$ is said to be {\em generic} if $\pi$
contains the character $\Psi$, or, equivalently, if there exists a nonzero map from
$\overline{\Gamma}$ to $\pi$.   The irreducible generic representations of $\overline{G}$
over $\overline{\CK}$ are indexed by semisimple conjugacy classes $s$ in $\overline{G}'$,
where $\overline{G}'$ is the group of $\overline{\FF}_q$-points in
the group $\CGbar'$ that is dual to $\CGbar$.  More precisely, given such an
$s$, there exists a unique irreducible generic representation $\overline{\St}_s$ in the
rational series attached to $s$.  

The association of rational series to semisimple conjugacy classes in $\overline{G}'$ depends
on choices which we now recall: let $\mu^{(p)}$ denote the prime-to-$p$ roots of unity in $\overline{\CK}$,
let $(\QQ/\ZZ)^{(p)}$ denote the elements of order prime to $p$ in $(\QQ/\ZZ)$, and fix isomorphisms:
$$\mu^{(p)} \cong (\QQ/\ZZ)^{(p)} \cong \overline{\FF}_q^{\times}.$$

Now let $t$ be a semisimple element in $\overline{G}'$, let $\CTbar'$ be a maximal
torus containing $s$, and let $\CTbar$ be the dual torus in $\overline{G}$.  Let
$X$ and $X'$ denote the character groups of $\CTbar$ and $\CTbar'$ respectively.
We have isomorphisms:
$$\CTbar(\FF_q) \cong \Hom(X/(\Fr_q - 1)X, \GG_m)$$
$$\CTbar'(\FF_q) \cong \Hom(X'/(\Fr_q - 1)X', \GG_m)$$
where $\Fr_q$ is the endomorphism induced by the $q$-power Frobenius.
We also have a natural duality $X/(\Fr_q - 1)X \cong \Hom(X'/(\Fr_q - 1)X', (\QQ/\ZZ)^{(p)})$.
The identifications we fixed above then give rise to isomorphisms:
$$\CTbar'(\FF_q) \cong \Hom(X'/(\Fr_q - 1)X', \GG_m) \cong X/(\Fr_q - 1)X \cong \Hom(\CTbar(\FF_q), \mu^{(p)}).$$
In this way we associate, to any semisimple element $t$ of $\CGbar'(\FF_q)$, and any $\CTbar'$ containing $t$,
a character $\varphi_{\CTbar',t}: \CTbar(\FF_q) \rightarrow \overline{\CK}^{\times}$.

It is immediate (by applying the idempotent of $\overline{\CK}[\overline{G}]$ corresponding to the
rational series attached to $s$ to Theorem~\ref{thm:dudas}) that we then have: 
\begin{prop}
Let $\CTbar$ be a maximal torus of $\CGbar$, and let $\CBbar$ be a Borel containing $\CTbar$.
Then, up to a cohomological shift depending only on $\CBbar$, we have:
$$r^{\CTbar \subseteq \CBbar}_{\overline{G}} \overline{St}_s \cong \bigoplus_{t \sim s; t \in \overline{T}'} \varphi_{\CTbar',t}.$$
\end{prop}

Returning to $\overline{\Gamma}$, we have a direct sum decomposition:

$$\overline{\Gamma} \otimes \overline{\CK} \cong \bigoplus_s \overline{\St}_s$$

It follows immediately that the endomorphism ring of $\overline{\Gamma} \otimes \overline{\CK}$
is isomorphic to a product of copies of $\overline{\CK}$, indexed by the semisimple conjugacy
classes $s$ in $\overline{G}'$.  As the endomorphism ring $\End_{W(k)[\overline{G}]}(\overline{\Gamma})$
of $\overline{\Gamma}$ embeds in this product, we see immediately that
$\End_{W(k)[\overline{G}]}(\overline{\Gamma})$ is reduced and commutative.  

Indeed, it is not difficult to describe the maps $\Phi_{\overline{T}} \otimes \overline{\CK}$.
The isomorphism:
$$\overline{\Gamma} \otimes \overline{\CK} \cong \bigoplus_s \overline{St}_s$$
where $s$ runs over semisimple conjugacy classes in $\overline{G}'$,
gives rise to an isomorphism:
$$\overline{E}_{q,n} \otimes \overline{\CK} \cong \prod_s \overline{\CK}.$$
On the other hand we have a direct sum decomposition:
$$\overline{\CK}[\overline{T}] \cong \bigoplus_t \varphi_{\CTbar,t}$$
of $\overline{\CK}[\overline{T}]$-modules, and hence an algebra isomorphism:
$$\overline{\CK}[\overline{T}] \cong \prod_t \overline{\CK}.$$
It follows immediately from the previous paragraph that $\Phi_{\overline{T}}$
maps the factor of $\overline{\CK}$ of $\overline{E}_{q,n} \otimes \overline{\CK})$ corresponding to
$s$ identically to each factor of $\overline{\CK}[\overline{T}]$ that corresponds to a $t$ in
the $\overline{G}'$-conjugacy class $s$, and to zero in the other factors.
%%As a consequence, we deduce that the image of $\Phi_{\overline{T}}$ is invariant under the action
%%of the normalizer of $\overline{T}$ under conjugation.

Now let $\CTbar$ range over all tori in $\overline{G}$, and consider the product map:
$$\Phi: \overline{E}_{q,n} \rightarrow
\prod_{\CTbar} W(k)[\overline{T}].$$
For each pair $(\CTbar,\varphi)$, where $\varphi$ is a character $\overline{T} \rightarrow \overline{\CK}^{\times}$, we have a map:
$$\xi_{\CTbar,\varphi}: \prod_{\overline{T}} W(k)[\overline{T}] \rightarrow \overline{\CK}$$
given by composing the projection onto $W(k)[\overline{T}]$ with the map:
$\varphi: W(k)[\overline{T}] \rightarrow \CK$.

Define an equivalence relation on such pairs by setting $(\CTbar_1,\varphi_1) \sim (\CTbar_2,\varphi_2)$
if $t_1$ and $t_2$ are conjugate in $\overline{G}'$, where $t_1$ and $t_2$ are the elements of the dual
tori $\CTbar'_1$ and $\CTbar'_2$ corresponding to $\varphi_1$ and $\varphi_2$.  Then our description of each $\Phi_{\overline{T}}$
shows that, when $(\CTbar_1,\varphi_1) \sim (\CTbar_2,\varphi_2)$, one has
$\xi_{\CTbar_1,\varphi_1} \circ \Phi = \xi_{\CTbar_2,\varphi_2} \circ \Phi$.  
Thus $\Phi$ induces a bijection between the $\overline{\CK}$-points of $\Spec \overline{E}_{q,n}$
and the equivalence classes of pairs $(\CTbar,\varphi)$.

In what follows, it will be necessary for us to consider certain direct factors of $\overline{E}_{q,n}$ arising from
idempotents of $W(k)[\overline{G}]$.  An $\ell$-regular semisimple conjugacy class $s$ in $\overline{G}'$ gives rise, via
the choices we have made above, to an idempotent $e_s$ in $W(k)[\overline{G}]$, that acts by the identity
on the rational series corresponding to those $s'$ in $\overline{G}$ with $\ell$-regular part $s$, and zero elsewhere.
We will denote by $\overline{E}_{q,n,s}$ the direct factor $e_s \overline{E}_{q,n}$ of $\overline{E}_{q,n}$. 
The $\overline{\CK}$-points of $\Spec \overline{E}_{q,n,s}$ are those corresponding to pairs $(\CTbar,\varphi)$
such that $\varphi$ corresponds to an element $t$ of $\CTbar'$ whose $\ell$-regular part is $s$.

Now let $s \in \overline{G}'$ be $\ell$-regular semisimple and suppose that the characteristic polynomial of $s$ is a power of an irreducible polynomial
of degree $d$.  Then the centralizer $\CLbar'$ of $s$ in $\CGbar'$ is a nonsplit Levi isomorphic to $\Res_{\FF_{q^d}/\FF_q} \GL_{\frac{n}{d}}$.  Let $\CLbar$ be the
Levi of $\CGbar$ dual to $\CLbar'$.  By Th{\'e}or{\`e}me 11.8 of~\cite{BR}, twisting by the character of $\overline{L}$ associated to $s$, followed by Deligne-Lusztig induction
from $\CLbar$ to $\CGbar$, is an equivalence of categories from $e_1\Rep_{W(k)}(\overline{L})$ to $e_s\Rep_{W(k)}(\overline{G})$.  Moreover, this equivalence carries
$e_1 \overline{\Gamma}_{\CLbar}$ to $e_s \overline{\Gamma}$.  (This follows from uniqueness of projective envelopes,
since the former is the projective envelope of the unique irreducible generic $k$-representation of  
$\overline{L}$ in the block corresponding to $e_1$, and the latter is the projective envelope of the unique irreducible generic $k$-representation of $\overline{G}$
in the block corresponding to $e_s$.)  We thus have:

\begin{proposition} \label{prop:E jordan}
For $s$ an $\ell$-regular semisimple element of $\overline{G}'$ whose characteristic polynomial is a power of an irreducible polynomial of degree $d$ over $\FF_q$.
Then there is a natural isomorphism:
$$\overline{E}_{q,n,s} \cong \overline{E}_{q^d,\frac{n}{d},1}.$$
The induced map on $\overline{\CK}$-points takes the $\overline{\CK}$-point of $\Spec \overline{E}_{q^d,\frac{n}{d},1}$ corresponding to the $\ell$-primary conjugacy class
$t$ of $\overline{L}'$ to the $\overline{\CK}$-point of $\Spec \overline{E}_{q,n,s}$ corresponding to the conjugacy class of $st$ in $\overline{G}'$.
\end{proposition}
\begin{proof}
The first claim is immediate from the previous paragraph.  The second follows from the description of the equivalence of categories on irreducible generic $\overline{\CK}$-representations.
\end{proof}

The final structure we will need to consider on $\overline{E}_{q,n}$ is a natural symmetrizing form considered by
Bonnaf{\'e}-Kessar (\cite{gelfand-graev}, section 3.B).  Define a $W(k)$-linear map $\theta: \overline{E}_{q,n} \rightarrow W(k)$ by the formula:
$$\theta(x) = \frac{1}{n!} \sum\limits_{w \in S_n} \theta_w(\Phi_{\CTbar_w}(x)),$$
where $\CTbar_w$ is the torus of $\CGbar$ associated to the element $w$ of the Weyl group, and $\theta_w: W(k)[\overline{T}_w] \rightarrow W(k)$
is the canonical symmetrizing form on $W(k)[\overline{T}_w]$ given by ``evaluation at the identity''.  Note that we can extend $\theta$
to a linear map $\overline{E}_{q,n} \otimes \overline{\CK} \rightarrow \overline{\CK}$.

We then have:
\begin{proposition} \label{prop:E-symmetrizing}
Let $t$ be a semisimple conjugacy class in $\overline{G}'$, and let $e_t$ be the corresponding idempotent of $\overline{E}_{q,n} \otimes \overline{\CK}$.
Then 
$$\theta(e_t) = \frac{1}{n!} \sum\limits_{w \in S_n} \frac{1}{\# \overline{T}_w} N(w,t),$$
where $N(w,t)$ is the number of elements of $\overline{T}'_w$ in the conjugacy class of $t$.
\end{proposition}
\begin{proof}
It is easy to see that $\Phi_{\CTbar_w}(e_t)$ is equal to the sum, over those $t' \in \CTbar'_w$ conjugate to $t'$, of the idempotents $e_{t'}$
of $\overline{\CK}[\overline{T}_w]$.  The claim is then immediate from the formula for $\theta$.
\end{proof}

\section{The integral Bernstein center} \label{sec:bernstein}

We now turn to the first main object of interest in this paper: the integral
Bernstein center.  Let $G = \GL_n(F)$, and
denote by $\Rep_{W(k)}(G)$ (resp. $\Rep_{\overline{\CK}}(G)$) the category
of smooth $W(k)[G]$-modules (resp. the category of smooth $\overline{\CK}[G]$-modules.)

By the phrase ``integral Bernstein center'' we mean the center of the category
$\Rep_{W(k)}(G)$.  We recall what this means:

\begin{defn} The {\em center} of an Abelian category $\CA$ is the ring of
natural transformations $\Id_{\CA} \rightarrow \Id_{\CA}$, where $\Id_{\CA}$
denotes the identity functor on $\CA$.
\end{defn}

By definition, if $Z$ is the center of $\CA$, then specifying an element of $Z$
amounts to specifying an endomorphism of every object of $\CA$, such that the
resulting collection commutes with all arrows in $\CA$.  The center of $\CA$ is
thus a commutative ring that acts naturally on every object in $\CA$, and this action
is compatible with all morphisms in $\CA$.

Bernstein-Deligne~\cite{BD}, give a complete and explicit description
of the center $\tZ$ of $\Rep_{\overline{\CK}}(G)$.  We briefly summarize their results:
first, define an equivalence relation on pairs $(M,\tpi)$, where $M$ is a Levi of $G$
and $\pi$ is an irreducible supercuspidal representation of $M$ over $\overline{\CK}$
by declaring $(M_1,\tpi_1)$ to be {\em inertially equivalent} to $(M_2,\tpi_2)$ if $\tpi_1$
is $G$-conjugate to an unramified twist of $\tpi_2$.  One then has:

\begin{thm}[\cite{BD}, Proposition 2.10]
There is a bijection $(M,\tpi) \mapsto e_{(M,\tpi)}$ between inertial equivalence classes
of pairs $(M,\tpi)$ over $\overline{\CK}$ and primitive idempotents of $\tZ$, such that
for any irreducible smooth representation $\Pi$ of $G$ over $\overline{\CK}$
$e_{(M,\tpi)}$ acts via the identity on $\Pi$ if $\Pi$ has supercuspidal support in the
inertial equivalence class of $(M,\tpi)$, and by zero otherwise.
\end{thm}

The upshot is that $\tZ$ decomposes as an infinite product of the rings $e_{(M,\tpi)} \tZ$
as $(M,\tpi)$ runs over all inertial equivalence classes of pairs.  Denote
$e_{(M,\tpi)} \tZ$ by $\tZ_{(M,\tpi)}$.  Then Bernstein and Deligne give a complete
description of the ring structure of $\tZ_{(M,\tpi)}$ that we now explain.

Let $M_0$ be the smallest subgroup of $M$ containing every compact open subgroup of $M$.  Then
$M/M_0$ is a free abelian group of finite rank, and $\Spec \overline{\CK}[M/M_0]$ is
a torus whose $\overline{\CK}$-points are in bijection with the characters
$M/M_0 \rightarrow \overline{\CK}^{\times}$.  Let $H$ be the subgroup of these characters consisting
of those characters $\chi$ such that $\tpi \otimes \chi$ is isomorphic to $\tpi$.  Then $H$ is a finite
abelian group that acts on $\overline{\CK}[M/M_0]$.  The
torus $\Spec \overline{\CK}[(M/M_0)]^H$ is a quotient of $\Spec \overline{\CK}[M/M_0]$; its $\overline{\CK}$-points
correspond to $H$-orbits of characters of $M/M_0$.

Now let $W_M$ be the subgroup of the Weyl group of $G$ consisting of $w$ such that $wMw^{-1} = M$.  Let
$W_M(\tpi)$ be the subgroup of $W_M$ consisting of $w$ such that the representation $\tpi^w$ of $M$ is
an unramified twist of $\tpi$.  Then we have a natural action of $W_M(\tpi)$ on $\overline{\CK}[(M/M_0)]^H$,
characterized by $\tpi \otimes \chi^w \cong (\tpi \otimes \chi)^w$ for characters $\chi$
of $M/M_0$.  We then have:

\begin{theorem}[\cite{BD}, Th{\'e}or{\`e}me 2.13]
There is a unique natural isomorphism:
$$\tZ_{(M,\tpi)} \cong \left(\overline{\CK}[(M/M_0)]^H\right)^{W_M(\tpi)}$$
such that, for any irreducible representation $\Pi$ over $\overline{\CK}$ whose supercuspidal
support has the form $\tpi \otimes \chi$, $\tZ_{(M,\tpi)}$ acts on $\Pi$ via the map:
$$\left(\overline{\CK}[(M/M_0)]^H\right)^{W_M(\tpi)} \rightarrow \overline{\CK}[M/M_0] \rightarrow \overline{\CK}$$
corresponding to the character $\chi: M/M_0 \rightarrow \overline{\CK}^{\times}$.  In particular
$\tZ_{(M,\tpi)}$ is a reduced, finitely generated, and normal $\overline{\CK}$-algebra.
\end{theorem}

In particular, $\tZ$ acts on two irreducible representations $\Pi$, $\Pi'$ of $G$ via the same
map $\tZ \rightarrow \overline{\CK}$ if, and only if, $\Pi$ and $\Pi'$ have the same supercuspidal support.
This defines, for each $(M,\tpi)$, a bijection between the $\overline{\CK}$-points of $\Spec \tZ_{(M,\tpi)}$
and supercuspidal supports in the inertial equivalence class of $(M,\tpi)$; that is, unramified twists
of $\tpi$ considered up to $W_M(\tpi)$-conjugacy.

Now let $L$ be a Levi in $\GL_n$; then $L$ factors as a product of $L_i$ isomorphic to $\GL_{n_i}(F)$.
For each $i$, let $M_i$ be a Levi in $L_i$, and $\tpi_i$ an irreducible supercuspidal $\overline{\CK}$-representation
of $M_i$.  We then have isomorphisms:
$$\tZ_{M_i,\tpi_i} \cong \left(\overline{\CK}[(M_i/(M_i)_0)]^{H_i}\right)^{W_{M_i}(\tpi_i)}.$$
Let $M$ be the product of the $M_i$; we may regard it as a Levi of $L$ and hence as a Levi of $\GL_n(F)$.
Let $\tpi$ be the tensor product of the $\tpi_i$.  The quotient $M/M_0$ factors naturally as a product
of $M_i/(M_i)_0$, and this induces a map:
$$\left(\overline{\CK}[(M/M_0)]^H\right)^{W_M(\tpi)} \rightarrow \bigotimes_i \left(\overline{\CK}[(M_i/(M_i)_0)]^{H_i}\right)^{W_{M_i}(\tpi_i)}$$
and hence a map 
$$\Ind_{\{(M_i,\tpi_i)\}}: \tZ_{(M,\tpi)} \rightarrow \bigotimes_i \tZ_{(M_i,\tpi_i)}.$$
On $\overline{\CK}$-points
this takes the $\CK$-point of the tensor product that corresponds to the collection of supercuspidal supports
$\{(M_i,\tpi_i \otimes \chi_i)\}$ to the point of $\Spec \tZ_{(M,\tpi)}$ corresponding to the supercuspidal support
$(M,\otimes_i (\tpi_i \otimes \chi_i))$.

%%%%We will have need of the following, nearly trivial, observation about this map:
%%%%\begin{lemma} \label{lemma:char 0 saturation}
%%%%The map $\tZ_{(M,\tpi)} \rightarrow \bigotimes_i \tZ_{(M_i,\tpi_i)}$ is $\overline{\CK}$-saturated.
%%%%\end{lemma}
%%%%\begin{proof}
%%%%Embed both $\tZ_{(M,\tpi)}$ and $\bigotimes_i \tZ_{(M_i,\tpi_i)}$ in $\bigotimes_i \overline{\CK}[M_i/(M_i)_0]$
%%%%as above.  The target has an action of $W_M(\tpi)$, and hence also an action of the subgroup $\prod_i W_{M_i}(\tpi_i)$,
%%%%and $\tZ_{(M,\tpi)}$ and $\bigotimes_i \tZ_{(M_i,\tpi_i)}$ are identified with those elements invariant
%%%%under $W_M(\tpi)$ or this subgroup, respectively.  But one can check whether an element in $\bigotimes_i \overline{\CK}[M_i/(M_i)_0]$
%%%%is invariant simply by looking at its values on $\overline{\CK}$-points, so the result is immediate.
%%%%\end{proof}

We now turn to the study of $\Rep_{W(k)}(G)$; let $Z$ denote the center of this category.
In this setting there is an analogue of the Bernstein-Deligne characterization of the primitive idempotents
of $Z$.  By~\cite{bernstein1}, Theorem 11.8, such idempotents are parameterized by inertial equivalence classes of pairs $(L,\pi)$, where
$\pi$ is now an irreducible supercuspidal representation of $L$ over $k$.

If we let $e_{[L,\pi]}$ denote the idempotent of $Z$ corresponding to $(L,\pi)$, $\Rep_{W(k)}(G)_{[L,\pi]}$ the corresponding
block, and $Z_{[L,\pi]}$ the
corresponding factor of the Bernstein center, then one has the following basic structure results:

\begin{theorem}[\cite{bernstein1}, Theorem 12.8]
The ring $Z_{[L,\pi]}$ is a finitely generated, reduced, flat $W(k)$-algebra.
\end{theorem}

It is important to note that, in contrast to the situation over $\overline{\CK}$, the ring $Z_{[L,\pi]}$
is in general very far from being normal.

We also have a description of $Z_{[L,\pi]} \otimes \overline{\CK}$
in terms of $\tZ$.  This can be made precise as follows: if $(M,\tpi)$ is a pair over $\overline{\CK}$,
and $\Pi$ is an irreducible integral representation of $G$ over $\overline{\CK}$ with supercuspidal support
in the inertial equivalence class of $(M,\tpi)$, then there exists a (possibly proper) Levi subgroup
$L$ of $M$, and an irreducible supercuspidal representation $\pi$ of $L$, such that every irreducible
subquotient of the mod $\ell$ reduction of $\Pi$ has supercuspidal support $(L,\pi)$.  Moreover, the inertial
equivalence class of $(L,\pi)$ depends only on that of $(M,\tpi)$, and not on the particular choice of $\pi$.
We say that $(M,\tpi)$ {\em reduces modulo $\ell$} to $(L,\pi)$; this defines a finite-to-one map from
inertial equivalence classes over $\overline{\CK}$ to inertial equivalence classes over $k$.  One then has:

\begin{theorem}[\cite{bernstein1}, Proposition 12.1] \label{thm:bernstein char 0}
The natural map $Z \otimes \overline{\CK} \rightarrow \tZ$ induces an isomorphism:
$$Z_{[L,\pi]} \otimes \overline{\CK} \cong \prod_{(M,\tpi)} \tZ_{(M,\tpi)},$$
where the product is over all pairs $(M,\tpi)$, up to inertial equivalence, that reduce modulo $\ell$
to the pair $(L,\pi)$.
\end{theorem}

From this and the description of the $\overline{\CK}$-points of $\Spec \tZ_{(M,\tpi)}$ one immediately deduces:
\begin{corollary} \label{cor:bernstein points}
The $\overline{\CK}$-points of $\Spec Z_{[L,\pi]}$ are in bijection with the supercupidal
supports of irreducible smooth $\overline{\CK}$-representations in $\Rep_{W(k)}(G)_{[L,\pi]}$.
\end{corollary}
We now give a more precise description of $Z_{[L,\pi]}$.  We first reduce to a more easily studied
special case:

\begin{defn} A pair $(L,\pi)$ is {\em simple} if there exist $r,m$ such that $n = rm$,
$L$ is isomorphic to $\GL_m(F)^r$, and $\pi$, up to unramified twist, is of the form $(\pi')^{\otimes r}$
for an irreducible supercuspidal representation $\pi'$ of $\GL_m(F)$.
\end{defn}

Note that any pair $(L,\pi)$ factors uniquely as a product of simple pairs $(L^i,\pi^i)$, with
$\pi^i \cong (\pi'_i)^{\otimes r_i}$, such that no $\pi_i'$ is an unramified twist of any other.
One then has:

\begin{theorem}[\cite{bernstein1}, Theorem 12.4]
Let $\{(L^i,\pi^i)\}$ be the natural decomposition of $(L,\pi)$ as a product of simple pairs.
Then there is a natural isomorphism:
$$Z_{[L,\pi]} \cong \bigotimes_i Z_{[L^i,\pi^i]}$$
such that, for any sequence $\{(M^i,\tpi^i)\}$ reducing modulo $\ell$ to $\{(L^i,\tpi^i)\}$,
the diagram:
$$
\begin{array}{ccc}
Z_{[L,\pi]} \otimes \overline{\CK} & \rightarrow & [\bigotimes_i Z_{[L^i,\pi^i]}] \otimes \overline{\CK}\\
\downarrow & & \downarrow\\
\tZ_{(M,\tpi)} & \rightarrow & \bigotimes_i \tZ_{(M^i,\tpi^i)}
\end{array}
$$
commutes, where $(M,\tpi)$ is the product of the $(M_i,\tpi_i)$, and the bottom horizontal map is the
map $\Ind_{\{(M^i,\tpi^i)\}}$ described above.
\end{theorem}

We thus focus our attention on the case where $(L,\pi)$ is simple.  Fix an integer $n_1$ and an irreducible supercuspidal
representation $\pi'$ of $\GL_{n_1}(F)$ over $k$.  For each $m > 0$, let $L_m$ be a Levi of $\GL_{n_1m}(F)$
isomorphic to $\GL_{n_1}(F)^m$, and let $\pi_m$ be the representation $(\pi')^{\otimes m}$ of $L_m$.
We can then consider the family of rings $Z_m := Z_{[L_m,\pi_m]}$ as $n$ varies.

Section 13 of~\cite{bernstein1} contains detailed information about the structure of the family $Z_m$.  In
particular this structure theory is closely related to the endomorphism rings of certain projective
objects $\CP_{K_m,\tau_m}$ for particular $m$.  More precisely, consider 
the group of unramified characters $\chi$ of $\GL_{n_1}(F)$ such that $\pi' \otimes \chi$ is isomorphic to $\pi'$.
This is a finite group; denote its order by $f'$.  Then attached to the system of pairs $(L_m,\pi_m)$ we have
a system of projective objects $\CP_{K_m,\tau_m}$, where $m$ lies in the set 
$\{1, e_{q^{f'}}, \ell e_{q^{f'}}, \ell^2 e_{q^{f'}}, \dots \}$.  (We refer the reader to Sections 7 and 9 of~\cite{bernstein1}
for a construction and structure theory of these objects.)  For brevity, denote the representation
$\CP_{K_m,\tau_m}$ by $\CP_m$.

For such $m$, let $E_m$ denote the endomorphism ring of $\CP_m$.  Then, by Corollary 9.2 of~\cite{bernstein1},
$E_m$ is a reduced, finite type,
$\ell$-torsion free $W(k)$-algebra.  Moreover, we have a map $Z_m \rightarrow E_m$ that gives the action of
$Z_m$ on the object $\CP_m$ of $\Rep_{W(k)}(\GL_{n_1m}(F))_{[L_m,\pi_m]}$.

If $m$ is arbitrary, the relationship between the rings $Z_m$ and $E_m$ is more complicated.
For a partition $\nu$ of $m$, we will say that $\nu$ is $q$-relevant if each $\nu_i$ belongs to the set
$\{1, e_q, \ell e_q, \ell^2 e_q, \dots \}$,
where $e_q$ is the multiplicative order of $q$ modulo $\ell$
(relevant partitions were called admissible in~\cite{bernstein1}).
Let $\nu$ be the maximal $q^{f'}$-relevant partition of $m$. 
Let $M_{\nu}$ and $P_{\nu}$ be the standard Levi and (upper triangular) parabolic subgroups of $\GL_{n_1m}$ attached to $n_1\nu$,
so that $M_{\nu}$ is a product of $\GL_{n_1\nu_i}(F)$, and consider the representation $\bigotimes_i \CP_{\nu_i}$ of
$M_{\nu}$.  Then $Z_m$ acts on the parabolic induction $i_{P_{\nu}}^{\GL_{n_1m}(F)} \bigotimes_i \CP_{\nu_i}$,
and we have:

\begin{theorem}[\cite{bernstein1}, Theorem 13.7] \label{thm:bernstein ind}
The action of $Z_m$ on
$i_{P_{\nu}}^{\GL_{n_1m}(F)} \bigotimes_i \CP_{\nu_i}$ factors through
the action of $\bigotimes_i E_{\nu_i}$ on $\bigotimes_i \CP_{\nu_i}$.  Moreover, the resulting
map:
$$Z_m \rightarrow \bigotimes_i E_{\nu_i}$$
is injective with saturated image, and is an isomorphism if $m$ lies in
$\{1,e_{q^{f'}},\ell e_{q^{f'}}, \dots \}$.  (Note that in this case $\nu$ is the one-element partition
$\{m\}$ of $m$.)
\end{theorem}

For $m$ in $\{1, e_{q^{f'}}, \ell e_{q^{f'}}, \dots\}$ we thus have a natural
identification of $Z_m$ with $E_m$.  For arbitrary $m$, we can regard the map
$Z_m \rightarrow \bigotimes_i E_{\nu_i}$ as a map $Z_m \rightarrow \bigotimes_i Z_{\nu_i}$.  Denote this map
by $\Ind_{\nu}$.  It is injective with saturated image. 

For $m$ in $\{1, e_{q^{f'}}, \ell e_{q^{f'}}, \dots \}$, the results of Sections 7 and 9 of~\cite{bernstein1} give very
precise information about $E_m$, and hence $Z_m$.  In particular there is an integer $f$ dividing $f'$, 
and a
a cuspidal $k$-representation $\sigma_m$ of $\GL_{\frac{mf'}{f}}(\FF_{q^f})$ (attached to an $\ell$-regular conjugacy
class $(s'_1)^m$ with $s'_1$ irreducible of degree $f'$ over $\FF_{q^f}$),
such that the projective $\CP_m$ is a compact induction
$\cInd_{K_m}^{\GL_{n_1m}(F)} \tkappa_m \otimes \CP_{\sigma_m}$, where $\tkappa_m$ comes from type theory
and $\CP_{\sigma_m}$ is the projective envelope of $\sigma_m$, inflated to a representation of $K_m$ via a natural map
$K_m \rightarrow \GL_{\frac{mf'}{f}}(\FF_{q^f})$.

Section 5 of~\cite{bernstein1} shows that $\CP_{\sigma_m}$ is the projection of the Gelfand-Graev representation
of $\GL_{\frac{mf'}{f}}(\FF_{q^f})$ to the block containing $\sigma_m$.  In particular, the results of
section~\ref{sec:finite} identify the endomorphisms of $\CP_{\sigma_m}$ with $\overline{E}_{q^{f},md,s}$, where we have written $s = (s_1)^m$
and $d = \frac{f'}{f}$.  By Proposition~\ref{prop:E jordan} we may identify $\overline{E}_{q^f,md,s}$ with $\overline{E}_{q^{f'},m,1}$.

We thus obtain an embedding of $\overline{E}_{q^{f'},m,1}$ in $E_m$ for such $m$.  Furthermore, section 9 of~\cite{bernstein1}
constructs an invertible element $\Theta_{m,m}$ of $E_m$.  We thus obtain a map
$$\overline{E}_{q^{f'},m,1}[T,T^{-1}] \rightarrow E_m$$
taking $T$ to $\Theta_{m,m}$.  It follows easily from the description of $E_m$ as a Hecke algebra in section 9 of~\cite{bernstein1}
that the image of this map consists of the elements of $E_m$ supported on double cosets of the form $K_m z_{m,m}^r K_m$ for various $r$.
(In particular, this image is saturated in $E_m$.)

The image of $\overline{E}_{q^{f'},m,1}$ in $Z_m$ is easy to describe.  Indeed, we have:

\begin{prop} \label{prop:finite bernstein image}
Let $m$ lie in $\{1,e_{q^{f'}}, \ell e_{q^{f'}}, \dots \}$, and let $x$ be an element of $\overline{E}_{q^{f'},m,1}$, where the
latter is considered as a subalgebra of $Z_m$.  Then for any irreducible $\overline{\CK}$-representations
$\Pi,\Pi'$ of $\GL_{n_1m}(F)$ in the same block of $\Rep_{\overline{\CK}}(\GL_{n_1m}(F))$, the action of $x$
on $\Pi$ and $\Pi'$ is via the same scalar.  Conversely, any element of $Z_m$ with this property lies in $\overline{E}_{q^{f'},m,1}$.
\end{prop}
\begin{proof}
The ring $Z_m$ annihilates both $\Pi$ and $\Pi'$ unless $\Pi$ and $\Pi'$ belong to a block of the form $\Rep_{\overline{\CK}}(\GL_{n_1m}(F))_{(M_s,\pi_s)}$
for a suitable $s$, in the notation of~\cite{bernstein1}, section 9.  In this case the action of $Z_m$ on $\Pi$ and $\Pi'$ factors
through the action of $Z_m$ on the summand $\cInd_{K_m}^{\GL_{n_1m}(F)} \tkappa_m \otimes \St_s$ of 
$\cInd_{K_m}^{\GL_{n_1m}(F)} \tkappa_m \otimes \CP_{\sigma_m} \otimes \overline{\CK}$.  In particular the action of $x$ on $\Pi$ and $\Pi'$
factors through the action of $x$ on $\St_s$, which is by a scalar.

Since $\overline{E}_{q^{f'},m,1}$ is saturated in $Z_m$, it suffices to prove the converse over $\overline{\CK}$.  But it follows easily from our factorization
of $Z_m$ in characteristic zero that every idempotent of $Z_m \otimes \overline{\CK}$ is contained in $\overline{E}_{q^{f'},m,1}$; since these idempotents
correspond to the blocks of $\Rep_{\overline{\CK}}(\GL_{n_1m}(F))_{M,\pi}$ the claim follows.
\end{proof}

We also make the following observation about the action of $\Theta_{m,m} \in Z_m$:
\begin{prop} \label{prop:theta action}
Let $P$ be a parabolic subgroup of $\GL_{n_1m}(F)$, with Levi subgroup $M$,
and let $\pi$ be an irreducible cuspidal $\overline{\CK}$-representation of $M$ such that
$i_P^G \pi$ lies in the block corresponding to $L_m,\pi_m$.  Suppose that
$M$ decomposes as a product of groups $M_i = \GL_{n_1m_i}(F)$, and let $\chi$ be an unramified character of $M$,
of the form $\otimes_i (\chi_i \circ \det)$, where we regard $(\chi_i \circ \det)$ as a character of $M_i$.

Let $x \in \overline{\CK}^{\times}$ be the scalar by which $\Theta_{m,m}$ acts on
$i_P^G \pi$.  Then $\Theta_{m,m}$ acts on $i_P^G \pi \otimes \chi$ via $x \prod_i \chi_i^{f'}(\unif_F)$.
\end{prop}
\begin{proof}
For some $s$, the pair $(M,\pi)$ is conjugate to an unramified twist of one of the pairs $(M_s,\pi_s)$ described in section 9 of~\cite{bernstein1}.
Thus, by Theorem 9.4 of~\cite{bernstein1},
 the action of $\Theta_{m,m}$ on $\pi$ is via the element $\theta_{m,s}$ of $Z_{M_s,\pi_s}$ defined in section 9 of~\cite{bernstein1},
and the claim is immediate from the definition of $\theta_{m,s}$ in that section.
\end{proof}

Finally, let $m'$ and $m$ be two consecutive elements of $\{1,e_{q^{f'}}, \ell e_{q^{f'}}, \dots\}$, and
set $j = \frac{m}{m'}$.  Theorem 13.5 of~\cite{bernstein1} then provides a map:
$$\Ind_{m',m}: Z_m \rightarrow Z_{m'}^{\otimes j},$$
that is compatible with parabolic induction, in the sense that the action of $x$ in $Z_m$ on
$i_P^{\GL_{n_1m}(F)} \pi$ (where $P = MU$ is a parabolic such that $M$ is isomorphic to
$\GL_{n_1m'}(F)^j$) is induced by the action of $\Ind_{m',m}(x)$ on $\pi$.   The image of this map is not saturated
but we have:

\begin{theorem}[\cite{bernstein1}, Theorem 13.6] \label{thm:near saturation}
Let $y$ be an element of $Z_{m'}^{\otimes j}$ such that, for some $a$, $\ell^a y$ lies in
the image of $\Ind_{m',m}$.  Then there exists an element $\tilde y$ of $Z_m$, an element $x$
of $\overline{E}_{q^{f'},m,1}[T^{\pm 1}]$, and an integer $b > 0$ such that
$\Ind_{m',m}(x) = \ell^b(y - \Ind_{m',m}(\tilde y))$.
\end{theorem}

The map $\Ind_{m',m}$ is not injective, but its kernel has a rather simple structure:

\begin{proposition} \label{prop:kernel}
There exists an ideal $I_{m',m}$ of $\overline{E}_{q^{f'},m,1}$ such that the kernel of $\Ind_{m',m}$ is
equal to $I_{m',m}[\Theta_{m,m}^{\pm 1}]$.
\end{proposition}
\begin{proof}
Since $\overline{E}_{q^{f'},m,1}[\Theta_{m,m}^{\pm 1}]$ is saturated in $Z_m$ we can prove this after tensoring with $\overline{\CK}$.  We have
a decomposition 
$$Z_m \otimes \overline{\CK} \cong \prod_i \tZ_{(M_i,\tpi_i)}$$
where $(M_i,\tpi_i)$ run over the $\overline{\CK}$-inertial equivalence classes in the block corresponding to $[L_m,\pi_m]$.  In particular the partitions
corresponding to the $M_i$ are all $q^{f'}$-relevant.  Fix a factor in this product corresponding to a pair $(M_i,\tpi_i)$.  On this factor,
we can descrive the map $\Ind_{m',m}$ in the following way: let $(M_{ij},\tpi_{ij})$ run over the set of $M_{\nu}$-intertial equivalence classes of pairs
that are $\GL_{n_1m}$-inertially equivalent to $(M_i,\tpi_i)$, where $\nu$ is the partition $(m',\dots,m')$ of $m$ and $M_{\nu}$ is the corresponding Levi of $\GL_{n_1m}$.
Since $M_{ij}$ is a Levi contained in $M_{\nu}$ the pair $(M_{ij},\tpi_{ij})$ breaks up as a product of $\frac{m}{m'}$ pairs $(M_{ijk},\tpi_{ijk})$ in $\GL_{n_1m'}$.
On the factor $\tZ_{(M_i,\tpi_i)}$ of $Z_m \otimes \overline{\CK}$, $\Ind_{m,n}$ is the sum of the maps:
$$\Ind_{(M_{ij},\tpi_{ij})}: \tZ_{(M_i,\tpi_i)} \rightarrow \bigotimes_k \tZ_{(M_{ijk},\tpi_{ijk})}.$$
In particular $\Ind_{m',m}$ is injective on the factor $\tZ_{(M_i,\tpi_i)}$ if $M_i$ is a proper Levi subgroup and zero otherwise.  When $M_i$ is not a proper Levi then the pair
$(M_i,\tpi_i)$ gives a cuspidal inertial equivalence class, so $\tZ_{(M_i,\tpi_i)}$ is isomorphic to $\overline{\CK}[\Theta_{m,m}^{\pm 1}]$.  Thus the kernel of
$\Ind_{m',m} \otimes \overline{\CK}$ is equal to ${\tilde I}_{m',m}[\Theta_{m,m}^{\pm 1}]$, where ${\tilde I}_{m,m}$ is the ideal of $\overline{E}_{q^{f'},m,1} \otimes \overline{\CK}$
generated by the idempotents of the latter that correspond to cuspidal inertial equivalence classes $(M_i,\tpi_i)$.
\end{proof}

\section{The ring $R_{q,n}$} \label{sec:tame}

We now turn to the second principal object of study of this paper, which is a moduli space of representations of $W_F$.
We begin by studying spaces of tame representations.
Let $X_{q,n}$ be the affine $W(k)$-scheme
paramaterizing pairs of invertible $n$ by $n$ matrices $(\Fr,\sigma)$ such that $\Fr \sigma \Fr^{-1} = \sigma^q$, and let
$X_{q,n}^{0}$ be the connected component of $X_{q,n}$ containing the $k$-point $\Fr = \sigma = \Id_n$.  Let
$S_{q,n}$ (resp. $R_{q,n}$) be the ring of functions on $X_{q,n}$ (resp. $X_{q,n}^{0}$), so that $X_{q,n} = \Spec S_{q,n}$
and $X_{q,n}^{0} = \Spec R_{q,n}$.

\begin{lemma} \label{lemma:ell torsion}
Let $L$ be an algebraically closed field that is a $W(k)$-algebra and $x$ be an $L$-point of $X_{q,n}$
corresponding to a pair $(\Fr_x,\sigma_x)$
of elements of $\GL_n(L)$.  Then $x$ lies in
$X_{q,n}^0$ if, and only if, the eigenvalues of $\sigma_x$ are $\ell$-power roots of unity.
\end{lemma}
\begin{proof}
Consider the map $X_{q,n} \rightarrow \AA^n_{W(k)}$ that takes a point $x$ to the coefficients
of the characteristic polynomial of $\sigma_x$.  Let $Y$ be the image of this map.  
For all $L$ and $x$, $\sigma_x$ is an element of $\GL_n(L)$ conjugate to its $q$th power, so its image in $Y(L)$
is a polynomial of degree $n$ whose roots, counted with multiplicities, are stable under the $q$th power map.  That is, every point
of $Y(L)$ corresponds to the characteristic polynomial of a diagonal matrix that is conjugate to its $q$th power.  Conversely,
given such a matrix $\sigma$ it is easy to construct an $L$-point $x$ of $X_{q,n}$ with $\sigma_x = \sigma$.

Let $\tY \subset \AA^n_{W(k)}$
be the space of diagonal matrices that are conjugate to their $q$th powers; we then have a map $\tY \rightarrow \AA^n_{W(k)}$
that sends such a matrix to the coefficients of its characteristic polynomial.  The argument of the previous paragraph shows that 
the (set-theoretic) image of $\tY$ is equal to $Y$.  On the other hand, $\tY(\overline{\CK})$ is a finite collection of points; indeed,
the entries of any diagonal matrix that is conjugate to its $q$th power are roots of unity of order bounded in terms of $q$ and $n$.  Thus the ``coordinates''
of each $\overline{\CK}$-point of $\tY$ are integral over $W(k)$, and every point of
$\tY(k)$ is in the closure of some point of $\tY(\overline{\CK})$.  It follows that the same is true for $Y$; in particular $Y$ is the
closure of a finite set of $\overline{\CK}$-points, and the closure of any $\overline{\CK}$-point of $Y$ meets the special fiber of $Y$.
Therefore, the connected component $Y^0$ of $Y$ containing the image of $X_{q,n}^0$ is the closure of the set of
$\overline{\CK}$-points of $Y$ that ``specialize'' mod $\ell$ to the characteristic polynomial $(X-1)^n$ of the identity matrix.
The only $k$-point of this component arises from the characteristic polynomial of the identity matrix, and the $\overline{\CK}$-points
of this component correspond to characteristic polynomials of elements of $\tY(\overline{\CK})$ whose roots reduce to $1$ modulo $\ell$.
The roots of such a polynomial are $\ell$-power roots of unity.  Therefore, for $x$ in $X_{q,n}^0(L)$ the roots of the characteristic
polynomial of $\sigma_x$ are $\ell$-power roots of unity, as required.

Conversely, let $x$ be an $L$-point of $X_{q,n}$, and suppose that the eigenvalues of $\sigma_x$ are $\ell$-power
roots of unity.  Note that $\GL_n(L)$ acts on $X_{q,n}(L)$, by conjugation on both $F$ and $\sigma$,
and this action preserves the connected components.  We may thus assume $\sigma_x$ is in Jordan normal form; in particular
its entries lie in $k$ or an integral extension $\OO$ of $W(k)$.  Moreover, for a fixed $\sigma_x$,
the set of $\Fr_x$ such that $\Fr_x \sigma_x = \sigma_x^q \Fr_x$
is a linear space; there is thus an invertible $\Fr'_x$ whose entries lie in $k$ or $W(k)$, such that $\Fr'_x \sigma_x = \sigma^q_x \Fr'_x$
and $(\Fr'_x,\sigma_x)$ lies on the same connected component as $x$.

If $L$ has characteristic $\ell$, the above construction yields a $k$-point of $X_{q,n}$ in the same connected component
as $x$.  If $L$ has characteristic zero, the closure of the point $(\Fr',\sigma)$ constructed above contains a $k$-point
$(\Fr'',\sigma')$ of $X_{q,n}$ in the same connected component as $x$.  Moreover, $\sigma'$ is unipotent and in Jordan normal form.
Thus in the closure of orbit of $(\Fr'',\sigma')$ under conjugation by
diagonal matrices there is a point where $\sigma$ is the identity.  It is clear that such a point lies in the connected
component of the $k$-point $x$ where $\Fr_x = \sigma_x = \Id_n$.
\end{proof}

The ring $R_{q,n}$ is rather well-behaved from an algebraic standpoint.  In particular, one has:
\begin{proposition}
The ring $R_{q,n}$ is reduced and locally a complete intersection.  Moreover, $R_{q,n}$
is flat as a $W(k)$-algebra.
\end{proposition}
\begin{proof}
This argument is a slight elaboration of an argument due to Choi~\cite{choithesis}.
We give a sketch here.

First note that $X_{q,n}$ is given by $n^2$ relations in a space of dimension $2n^2 + 1$.
Consider the map $X_{q,n} \rightarrow \AA^{n^2}_{W(k)}$ that sends a point $x$ to
the matrix $\sigma_x$.
Let $L$ be an algebraically closed field that is a $W(k)$-algebra, and let $x$ be an $L$-point
of $X_{q,n}$. 

The group $\GL_n(L)$ acts on the set of $L$-points of $X_{q,n}$ by conjugation.
Consider the locally closed subset $U_{\sigma_x}$ of $\Spec \AA^{n^2}_L$ consisting of those
$\sigma'$ conjugate to $\sigma_x$.  For any $L$-point $\sigma'$ of $U_{\sigma_x}$, the fiber
of $X_{q,n} \times_{W(k)} L$
over $\sigma'$ consists of pairs $(\Fr' h, \sigma')$, where $\Fr'$ is a fixed element of $\GL_n$ such
that $\Fr' \sigma' (\Fr')^{-1} = (\sigma')^q$ and $h$ commutes with $\sigma'$

In particular, the dimension
of the preimage of $U_{\sigma}$ in $X_{q,n} \times_{W(k)} L$ is equal to
the dimension of $U_{\sigma}$ plus the dimension of the stabilizer of $\sigma$ under conjugation;
this is clearly $n^2$.   As $\sigma$ varies over a finite list of conjugacy classes, the
preimages of the $U_{\sigma}$ cover $X_{q,n} \times_{W(k)} L$; thus $X_{q,n} \times_{W(k)} L$
is equidimensional of dimension $n^2$.  On the other hand the dimension of $X_{q,n}$ is at least
$n^2 + 1$.  It follows that the Zariski closures of the preimages of sets $U_{\sigma}$ are irreducible components of
$X_{q,n}$, and that no irreducible component of $X_{q,n}$ is contained in the special fiber (as it
would then be a component of $X_{q,n} \times_{W(k)} k$ of dimension at most $n^2$).  It also
follows that every irreducible component of $X_{q,n}$ has dimension $n^2 + 1$, because if we had
a component of larger dimension then its base change to $\overline{\CK}$ would have dimension
greater than $n^2$.  In particular $X_{q,n}$ is a complete intersection.  It follows that $R_{q,n}$
is a local complete intersection.

An argument of Choi (\cite{choithesis}, Theorem 3.0.13) shows that for any maximal ideal
$m$ of $R_{q,n}$, $(\Spec R_{q,n})_m[\frac{1}{\ell}]$ is
generically smooth; in particular $X^0_{q,n}$ is generically reduced.
By the unmixedness theorem the local complete intersection $X^0_{q,n}$ has no embedded points, so
$R_{q,n}$ is reduced.  As the generic points of $\Spec R_{q,n}$ all have characteristic zero,
we may conclude that $R_{q,n}$ is flat over $W(k)$.
\end{proof}

We have a universal pair of matrices $(\Fr,\sigma)$ in $\GL_n(R_{q,n})$.  The above result immediately implies:

\begin{corollary} \label{cor:continuity}
There exists a power $\ell^a$ of $\ell$ such that
$\sigma^{\ell^a}$ is unipotent in $\GL_n(R_{q,n})$.
\end{corollary}
\begin{proof}
Since $R_{q,n}$ is reduced and flat over $W(k)$, it suffices to check that $\sigma^{\ell^a}$
is unipotent for some $a$ at each of the generic points of $\Spec R_{q,n}$, all of which lie in characteristic zero.
This is an immediate consequence of Lemma~\ref{lemma:ell torsion}.
\end{proof}

Let $L$ be a finite extension of $\CK$.  We call an $L$-point of $X^0_{q,n}$ {\em integral}
if the corresponding map $R_{q,n} \rightarrow L$ factors through the ring of integers $\OO_L$.

\begin{lemma} \label{lemma:integral}
Let $x$ be an $L$-point of $X^0_{q,n}$, and suppose that the eigenvalues of
$\Fr_x$ lie in $\OO_{L'}^{\times}$ for some finite extension $L'$ of $L$.  
Then there is an integral point of $X^0_{q,n}$ in the $\GL_n$-orbit of $x$.
\end{lemma}
\begin{proof}
Extending $L$ if necessary, we may assume that the eigenvalues of $\sigma_x$ are in $L$, and hence
$\OO_L$.  Then (for instance, by putting $\sigma_x$ in Jordan normal form) we can find an $\OO_L$-sublattice
$M$ of $L^n$ preserved by $\sigma_x$.  Using $\Fr_x \sigma_x \Fr_x^{-1} = \sigma_x^q$, we find that
$\Fr_x M$, $\Fr_x^2 M$, etc. are also preserved by $\sigma_x$.  Consider the lattice $M'$ given by
$M + \Fr_x M + \dots + \Fr_x^{n-1} M$; it is clearly preserved by $\sigma_x$.  On the other hand, since
$\Fr_x$ is annihilated by a polynomial with integral coefficients, $\Fr_x^n M$ is contained in $M'$,
and hence $\Fr_x M'$ is contained in $M'$.  Since $\Fr_x$ has unit determinant we must have $\Fr_x M' = M'$.
Thus $M'$ is stable under both $\Fr_x$ and $\sigma_x$.  Choosing a
basis for $M'$, we find an integral point of $X^0_{q,n}$ in the same $\GL_n$-orbit as $x$.
\end{proof}

\begin{lemma}
For any positive integer $m$, and any element $\lambda$ of $\OO_L^{\times}$, there is an element $g_{m,\lambda}$ of $\GL_m(L)$, with unit eigenvalues, such
that $g_{m,\lambda} J_{m,\lambda^q} g_{m,\lambda}^{-1} = J_{m,\lambda}^q$, where $J_{m,\lambda}$ is the unipotent Jordan block of size $m$.
\end{lemma}
\begin{proof}
The matrices $J_{m,\lambda^q}$ and $J_{m,\lambda}^q$ are regular with the same eigenvalues, hence conjugate by some $g' \in \GL_m(L)$.
Since $J_{m,\lambda^q}$ is contained in a unique Borel subgroup of $\GL_m$ (namely, the standard one), the same is true of $J_{m,\lambda^q}$.
Thus $g'$ normalizes the standard Borel, so $g'$ is upper triangular.  The eigenvalues of $g'$ are thus given by its diagonal entries $g'_1,\dots,g'_m$.
Comparing the $(i,i+1)$ entries of $J_{m,\lambda}^q$ and $J_{m,\lambda^q}$ we find that that $\frac{g'_{i+1}}{g'_i} = \lambda^{q-1} q$.  In
particular, multiplying $g'$ by a suitable scalar we may assume $g'$ has integral eigenvalues, as desired.
\end{proof}

\begin{prop} The images of the integral points of $X^0_{q,n}$ are dense in $X^0_{q,n}$.
\end{prop}
\begin{proof}

Fix a point $(\Fr_x,\sigma_x)$ of $X^0_{q,n}$.  
After conjugating $\sigma_x$ appropriately we may assume that $\sigma_x$ is in Jordan normal form (and thus in particular
has integral entries, since we have shown that the eigenvalues of $\sigma_x$ are roots of unity).  Moreover, since $\sigma_x$
is conjugate to its $q$th power, for any eigenvalue $\lambda$ of $\sigma$ there is a size-preserving bijection between
the Jordan blocks of $\sigma_x$ of eigenvalue $\lambda$ and those of eigenvalue $\lambda^q$.  Let $(m_i,\lambda_i)$ denote the
size and eigenvalue of the $i$th Jordan block of $\sigma_x$.  Then we can find a permutation matrix
$w$ such that $w \sigma_x w^{-1}$ is also in Jordan normal form, but where the $i$th Jordan block is of size $m_i$ with eigenvalue
$\lambda_i^q$.  Let $g$ be the block diagonal matrix whose $i$th block is the matrix $g_{m_i,\lambda_i}$ from the above lemma.
Then $g w \sigma_x (gw)^{-1} = \sigma_x^q$.  Moreover $gw$ has unit eigenvalues, as some power of $gw$ is block diagonal with
blocks given by powers of the matrices $g_{m_i,\lambda_i}$.  Thus by Lemma~\ref{lemma:integral} we can find an integral point $(\Fr'_x,\sigma'_x)$
of $X^0_{q,n}$ in the $\GL_n$-orbit of the point $(gw,\sigma_x)$.

Now consider the condition $g' \sigma'_x = \sigma'_x g'$, for arbitrary matrices $g'$.  This is a linear condition on $g'$ with
coefficients in $\OO_L$.  The scheme parameterizing such $g'$ is not quite a vector space scheme over $\OO_L$ (it need not be flat over $\OO_L$),
but the closure of its general fiber is such a scheme.  Let $U$ be the open subscheme of this closure consisting of
invertible $g'$.  Then $U$ contains the identity in particular, so its special fiber is nonempty.  However, in an open subset
of a vector space scheme over $\OO_L$ whose special fiber is nonempty, the $\OO_L$-points form a dense subset.  Thus integral
points are dense in $U$.

On the other hand, the points $(\Fr'_x u,\sigma'_x)$, as $u$ runs over
the integral points of $U$, are all integral points of $X^0_{q,n}$, and (since integral points of $U$ are dense in $U$)
their closure is the set of all points
$(\Fr'_y,\sigma'_x)$ in $X^0_{q,n}$.  Conjugating by integral points of $\GL_n$, which are clearly dense in $\GL_n$, we find that
the closure of the integral points contains the entire locus of points $(\Fr''_x, \sigma''_x)$ with $\sigma''_x$ conjugate to
$\sigma_x$.  Since $\sigma_x$ was chosen arbitrarily the result follows.
\end{proof}

\begin{cor} The ring $R_{q,n}$ is $\ell$-adically separated; that is, the intersection
of the ideals $\ell^i R_{q,n}$ is zero.
\end{cor}
\begin{proof}
Let $f$ be an element of $R_{q,n}$ that is divisible by $\ell^i$ for all $i$.  Then, for any integral point
$x: R_{q,n} \rightarrow \OO_L$, the image $x(f)$ is divisible by $\ell^i$ for all $i$ and is therefore zero.
In other words, $f$ vanishes on a dense subset of $X_{q,n}^0$.  Since $X_{q,n}^0$ is reduced, $f$ is zero.
\end{proof}

Now fix a Frobenius element $\tFr$ in $W_F$, and a topological generator $\tsigma$
of the quotient $I_F/I_F^{(\ell)}$.  Let $t_{\ell}$ be the isomorphism of $I_F/I_F^{(\ell)}$ with the additive
group of ${\mathbb Z}_{\ell}$ that takes $\tsigma$ to $1$.  By 
Corollary~\ref{cor:continuity}, for some positive integer $a$ the matrix $\sigma^{\ell^a}$ in $\GL_n(R_{q,n})$ 
is unipotent; that is, its characteristic polynomial is $(X - 1)^n$. The following lemma allows us to make sense
of $(\sigma^{\ell^a})^b$ for any $b \in \ZZ_{\ell}$:

\begin{lemma} Let $R$ be a flat, $\ell$-adically separated $\ZZ_{\ell}$-algebra, and $M \in \GL_n(R)$ such that $(M - 1)^n = 0$.  Then
there exists a unique $\ell$-adically continuous homomorphism $\phi_M: \ZZ_{\ell} \rightarrow \GL_n(R)$
such that for all $b \in \ZZ$, $\phi_M(b) = M^b$.
\end{lemma}
\begin{proof}
Consider the power series $\exp t \log (1 + X)$ in $\QQ[t][[X]]$, and let $p_i(t)$ be the coefficient of $X^i$ in
this power series.  For any $i$, and any integer $b$, Let $N_i$ be the $(i + 1)$ by $(i+1)$ Jordan block with eigenvalue zero; then 
$p_i(b)$ is the upper right entry of $(1 + N_i)^b$, and is thus an integer.  In particular each $p_i$ is a $\ZZ_{\ell}$-valued function
on $\ZZ_{\ell}$.  Given $M$ as above, and $t \in \ZZ_{\ell}$, we may thus define $\phi_M$ by 
$$\phi_M(t) = 1 + p_1(t) (M - 1) + \dots + p_{n-1}(t) (M - 1)^{n-1},$$
and it is clear that this has the claimed properties.
\end{proof}

(Recall that for an $\ell$-adically separated ring $A$, and a locally profinite group $H$, a representation $\rho: H \rightarrow \GL_n(A)$
is $\ell$-adically continuous if, for all positive integers $i$, the preimage of the subgroup $\Id + \ell^i M_n(A)$ of $\GL_n(A)$
is open in $H$.)

We will henceforth write $(\sigma^{\ell^a})^b$ for $\phi_{(\sigma^{\ell^a})}(b)$, given $b \in \ZZ_{\ell}$.

We thus have an $\ell$-adically continous representation
$\rho_{F,n}: W_F \rightarrow \GL_n(R_{q,n})$ defined by
$$\rho_{F,n}(\tFr^i w) = \Fr^i \sigma^j (\sigma^{\ell^a})^b,$$
for any $w \in I_F$ and any $j \in \ZZ$, $b \in \ZZ_{\ell}$ such that $j + \ell^a b = t_{\ell}(w).$
Note that, by the above lemma, this is the {\em unique} $\ell$-adically continuous representation
that takes $\tFr$ to $\Fr$ and $\tsigma$ to $\sigma$.

The pair $(R_{q,n}, \rho_{F,n})$ has the following universal property, which is easily seen to characterize
the pair up to isomorphism:

\begin{prop} \label{prop:tame universal}
For any finitely generated, $\ell$-adically separated $W(k)$-algebra $A$, and any framed, $\ell$-adically continuous
representation $\rho: W_F/I_F^{(\ell)} \rightarrow \GL_n(A)$, there is a unique map: $R_{q,n} \rightarrow A$
such that $\rho$ is the base change of $\rho_{F,n}$.
\end{prop}
\begin{proof}
Given $\rho$, we have a pair of matrices $(\rho(\tFr),\rho(\tsigma))$ in $\GL_n(A)$, satisfying
$$\rho(\tFr)\rho(\tsigma)\rho(\tFr)^{-1} = \rho(\tsigma)^q,$$ 
and hence a map $S_{q,n} \rightarrow A$.
Moreover, since the restriction of $\rho$ to $I_F$ factors through $I_F/I_F^{(\ell)}$ and is $\ell$-adically
continuous, the eigenvalues of $\rho(\tsigma)$ are $\ell$-power roots of unity.  Thus the map
$S_{q,n} \rightarrow A$ factors through $R_{q,n}$ and the result follows.
\end{proof}

If we regard the $\overline{\CK}$-points of $X_{q,n}^0$ as framed representations of $W_F/I_F^{(\ell)}$,
then one can show:

\begin{proposition} \label{prop:orbits}
Let $x$ be a $\overline{\CK}$-point of $X_{q,n}^0$.  Then there is a point $y$ in the closure of the $\GL_n$-orbit of $x$
such that the representation $\rho_y$ is semisimple.
\end{proposition}
\begin{proof}
Replacing $x$ with a point in the same $\GL_n$-orbit, we may assume that the framing on $\rho_x$ is such that $\rho_x$
is block upper triangular, with block sizes $n_1, \dots n_r$, and that for $1 \leq i \leq r$, the restriction $\rho_i$
of $\rho_x$ to the $i$th diagonal block is irreducible.  Let $M$ be the block diagonal matrix whose $i$th block is
given by $t^i$ times the $n_i$ by $n_i$ identity matrix, for some parameter $t$.  Then the limit, as $t$ approaches zero,
of $M \rho_x M^{-1}$ exists and is semisimple.
\end{proof}

We will later need the following observation about the representation $\rho_{F,n}$.

\begin{proposition} \label{prop:tame constant inertia}
As $x$ varies over the $\overline{\CK}$-points of $X_{q,n}^0$, the restriction of $\rho_x^{\sss}$ to $I_F$ is constant
on connected components of $X_{q,n}^0 \times_{W(k)} \overline{\CK}$.
\end{proposition}
\begin{proof}
The restriction of $\rho_x^{\sss}$ to $I_F$ is determined by the characteristic polyomial of $\sigma_x$;
since the eigenvalues of $\sigma_x$ have bounded $\ell$-power order there are only finitely possible characteristic polynomials of $\sigma_x$.
\end{proof}

\section{The inertial subalgebra of $S_{q,n}$} \label{sec:inertial}

Our next goal is to study the finite rank $W(k)$-subalgebra of $S_{q,n}$ generated by the coefficients of the characteristic polynomial
of $\sigma$.  Consider the map: 
$$W(k)[r_1,\dots,r_n,r_n^{-1}] \rightarrow S_{q,n}$$ 
that takes $r_i$ to the coefficient of $X^{n-i}$ in this characteristic polynomial.  

By the theory of symmetric functions, for $1 \leq i \leq n$ there are unique polynomials $P_{i,q}$ in the variables $r_1,\dots,r_n$
with the following property: for all $t_1,\dots,t_n \in \overline{\CK}$, define $r_1,\dots,r_n \in \overline{\CK}$ by the identity:
$$(X - t_1)\cdots(X - t_n) = X^n + r_1 X^{n-1} + r_2 X^{n-2} + \dots + r_n.$$
Then the $P_{i,q}$ are the unique polynomials satisfying:
$$(X - t_1^q)\cdots(X - t_n^q) = X^n + P_{1,q}(r_1,\dots,r_n) X^{n-1} + \dots + P_{n,q}(r_1,\dots,r_n).$$

Since $\sigma$ is conjugate to its $q$th power, for $1 \leq i \leq n$ the element
$P_{i,q}(r_1,\dots,r_n) - r_i$ lies in the kernel of the map $W(k)[r_1,\dots,r_n,r_n^{-1}] \rightarrow S_{q,n}$.  Let $I_{q,n}$
denote the ideal of $W(k)[r_1,\dots,r_n,r_n^{-1}]$ generated by the $P_{i,q}(r_1,\dots,r_n) - r_i$, and let $B_{q,n}$ denote
the quotient $W(k)[r_1,\dots,r_n,r_n^{-1}]/I_{q,n}$.  We will show that in fact the map $B_{q,n} \rightarrow S_{q,n}$ is injective, and
that moreover its image in $S_{q,n}$ is saturated.

We will now realize $B_{q,n}$ as a quotient of $S_{q,n}$ in a natural way.  We are grateful to Jack Shotton for making us aware
of the following construction, which is adapted from Proposition 7.10 in~\cite{shotton}.  (Shotton uses a slightly different form for the
matrix $\sigma$, that is less convenient for our purposes, but the arguments are otherwise exactly analogous.)

%We will now realize ${\tilde B}_{q,n}$ as a quotient of $S_{q,n}$ in a natural way.  We are grateful to Jack Shotton for making us aware
%of the following construction, which is adapted from Proposition 7.10 in~\cite{shotton}.

Let $Y \subseteq \Spec S_{q,n}$ denote the locus on which $\sigma$ has the form:
$$
\begin{pmatrix}
0 & 0 & 0 & \dots & 0 & -r_n\\
1 & 0 & 0 & \dots & 0 & -r_{n-1}\\
0 & 1 & 0 & \dots & 0 & 0\\
& \vdots & & & \vdots \\
0 & 0 & 0 & \dots & 1 & -r_1\\
\end{pmatrix}.
$$
(that is, on which $\sigma$ is the ``companion matrix'' of the polynomial $X^n + r_1 X^{n-1} + \dots + r_n$.)
We may embed $Y$ as an open subscheme of the scheme $Y'$ parameterizing pairs of matrices $(\Fr,\sigma)$ such that $\sigma$ is invertible of the above form, the characteristic
polynomial of $\sigma$ is equal to that of $\sigma^q$, and $\Fr \sigma = \sigma^q \Fr$.  Then $Y$ is simply the open
subscheme of $Y'$ on which $\Fr$ is invertible.  The scheme $Y'$ then maps to $\Spec B_{q,n}$ via the map that takes $(\Fr,\sigma)$ to the tuple $(r_1,\dots,r_n)$.

We have a map: $Y' \rightarrow \Spec B_{q,n} \times_{W(k)} \AA^n_{W(k)}$ that takes $(\Fr,\sigma)$ to the point $(r_1,\dots,r_n, \Fr(e_1))$,
where $e_1,\dots,e_n$ is the standard basis for $W(k)^n$.  In fact, one then has:

\begin{proposition} The map $Y' \rightarrow \Spec B_{q,n} \times_{W(k)} \AA^n_{W(k)}$ is an isomorphism.
\end{proposition}
\begin{proof}
We describe an inverse map.  Given $(r_1,\dots,r_n,v)$ in $\Spec B_{q,n} \times_{W(k)} \AA^n_{W(k)}$ we
associate the pair $(\Fr,\sigma)$, where $\sigma$ has the above form with $-r_n,\dots,-r_1$ in the right column, and $\Fr$ is defined by:
$\Fr e_i = \sigma^{(i-1)q}v$ for $1 \leq i \leq n$.  
One verifies easily that for $1 \leq i \leq n-1$, we have $\Fr\sigma(e_i) = \sigma^q\Fr(e_i)$.  On the other hand, we have:
$$\sigma^q \Fr e_n - \Fr \sigma e_n = (\sigma^{q})^n + r_1 (\sigma^q)^{n-1} + \dots + r_n)v  = P_{\sigma}(\sigma^q)v,$$
where $P_{\sigma}$ is the characteristic polynomial of $\sigma$.  The relations on ${\tilde B}_{q,n}$ guarantee that $P_{\sigma} = P_{\sigma^q}$,
so $P_{\sigma}(\sigma^q)v = 0$ by Cayley-Hamilton.

We thus have a well-defined map that is clearly a right inverse to the map constructed above.  To see that it is also a left inverse, note that if
$\Fr \sigma = \sigma^q \Fr$, and $\Fr(e_1) = v$, then we must have
$$\Fr e_i = \Fr\sigma(e_{i-1})= \sigma^q \Fr e_{i-1}$$
so by induction $\Fr$ is determined by $\Fr(e_1)$.
\end{proof}

\begin{lemma}
Let $B$ be a finite rank $W(k)$-algebra, and $V$ an open subset of $\Spec B \times_{W(k)} \AA^n_{W(k)}$ such that the projection $V \rightarrow \Spec B$
is surjective.  Then the map from $B$ to $\OO_V$ induced by the projection of $V$ onto $\Spec B$ is injective.  If moreover, $B$ is flat over $W(k)$,
then the image of $B$ in $\OO_V$ is saturated.
\end{lemma}
\begin{proof}
For each closed point $x$ of $\Spec B$, there exists an element $\overline{a}_x$ of $k^n$ such that $(x,\overline{a}_x)$ lies in $V$.  Lift
$\overline{a}_x$ to a $W(k)$-point $a_x$ of $\AA^n_{W(k)}$, and let $V_x = V \cap (\Spec B \times_{W(k)} a_x)$.  Then the projection
of $V$ to $\Spec V$ identifies $V_x$ with an open subset of $\Spec B$, and as $x$ varies, the $V_x$ cover $\Spec B$.  If $b$ is an element of $B$ that
maps to zero in $\OO_V$, then it vanishes in particular on each $V_x$ and hence on $\Spec B$, so injectivity is clear.

Now consider an element $b$ of $B/\ell B$, and suppose $B$ maps to zero in $\OO_V/\ell \OO_V$.  Then $b$ maps to zero in $\OO_{V_x}/\ell \OO_{V_x}$
for all $x$, but since the $V_x$ are an open cover of $\Spec B$ this means $b$ is zero in $B/\ell B$.
\end{proof}

We can now show:
\begin{proposition}  \label{prop:inertial saturation}
The map $B_{q,n} \rightarrow S_{q,n}$ is injective with saturated image.
\end{proposition}
\begin{proof}
We first show that the projection map from $Y$ to $\Spec B$ is surjective.  Indeed, for any algebraically closed field $L$ that is a $W(k)$-algebra,
and any $L$-point $(r_1,\dots,r_n)$ of $\Spec B$, the corresponding $\sigma$ is a regular element of $L$ whose characteristic polynomial is equal to that of $\sigma^q$.
In particular the eigenvalues of $\sigma$ are roots of unity of order prime to $q$.  It is then clear, by considering the Jordan normal form of $\sigma$, that $\sigma^q$
is also regular.  Over $L$ any two regular matrices with the same characteristic polynomial are conjugate, so there exists an element $\Fr$ of $\GL_n(L)$ that conjugates
$\sigma$ to $\sigma^q$.  Then $(\Fr,\sigma)$ is an $L$-point of $T$ mapping to $(r_1,\dots,r_n)$.

The lemma now shows that the map from $B_{q,n}$ to ${\mathcal O}_Y$ is injective; since this map factors through $S_{q,n}$ we see that $B_{q,n}$ embeds in $S_{q,n}$.
Thus $B_{q,n}$ is flat over $W(k)$, and the lemma then shows that its image in ${\mathcal O}_Y$ is saturated.  Once again using that the map from $B_{q,n}$ to ${\mathcal O}_Y$
factors through $S_{q,n}$ we see that $B_{q,n}$ is also saturated in $S_{q,n}$.
\end{proof}

The map $B_{q,n} \rightarrow S_{q,n}$ induces a map $B_{q,n,1} \rightarrow R_{q,n}$, where $B_{q,n,1}$ is the direct factor of $B_{q,n}$ whose $\overline{\CK}$-points
correspond to conjugacy classes whose reduction modulo $\ell$ is the identity.  Proposition~\ref{prop:tame constant inertia}, together with Proposition~\ref{prop:inertial saturation},
shows that $B_{q,n,1}$ is precisely the subalgebra of $R_{q,n}$ consisting of elements whose value at a $\overline{\CK}$-point $x$ of $\Spec R_{q,n}$ depends only on
the semisimplification of the restriction of $\rho_x$ to $I_F$.

\section{The symmetrizing form on $B_{q,n}$} \label{sec:E-B}

We now relate $B_{q,n}$ with the endomorphism ring $\overline{E}_{q,n}$ of the Gelfand-Graev representation.  We first work over $\overline{\CK}$;
since both $B_{q,n}$ and $\overline{E}_{q,n}$ are reduced, constructing an isomorphism of $B_{q,n} \otimes \overline{\CK}$ with $\overline{E}_{q,n} \otimes \overline{\CK}$
amounts to constructing a bijection on their $\overline{\CK}$-points.

Recall that the $\overline{\CK}$-points of $\Spec \overline{E}_{q,n}$ are in bijection with the isomorphism classes of irreducible generic representations of $\overline{G}$
and therefore (via Deligne-Lusztig restriction) with the equivalence classes of pairs $(w,\varphi)$ where $w$ is an element of the Weyl group of $\overline{G}$
and $\varphi: \overline{T}_w \rightarrow \overline{\CK}^{\times}$ is a character.  On the other hand, a $\overline{\CK}$-point of $\Spec B_{q,n}$ is represented
by an invertible diagonal matrix, with entries in $\overline{\CK}$, that is conjugate to its $q$-th power; that is, it is an invertible diagonal matrix $t$
such that there exists a permutation matrix $w$ with $t^w = t^q$.

In order to construct a natural bijection between these two sets we must fix some choices.  First, we
identify $\GL_n(\overline{\CK})$ with the Langlands dual group ${\hat G}$ of $G$, with (diagonal) maximal torus ${\hat T}$.
Second, we choose a topological generator $\tsigma$ of the tame inertia group $I_F/P_F$ of $F$.  Local class field theory gives an isomorphism:
$$I_F/P_F \cong \varprojlim \FF_{q^n}^{\times} \cong 
\varprojlim \Hom(\frac{1}{q^n-1}\ZZ/\ZZ, \overline{\FF}_q^{\times}),$$
where the first limit is over the norm maps, and the transition maps in the second limit, for $m$ dividing $n$, are given by ``multiplication by $\frac{q^n -1}{q^m - 1}$''.

On the other hand we have a chain of natural isomorphisms:
$$\Hom((\QQ/\ZZ)^{(p)}, \overline{\FF}_q^{\times}) = \Hom(\varinjlim (\frac{1}{q^n-1}\ZZ/\ZZ), \overline{\FF}_q^{\times}) \cong I_F/P_F,$$
so our choice of $\tsigma$ gives us a natural map $(\QQ/\ZZ)^{(p)} \rightarrow \overline{\FF}_q^{\times}$ that is easily seen to be an isomorphism.

Now fix a $w$ in the Weyl group $W(\overline{G})$; we identify $W(\overline{G})$ with the group of permutation matrices in $\GL_n(\overline{\CK})$. 
Let $X$ be the character group of the torus $\CTbar_w$ of $\CGbar$; then $X$ is dual to the character group $X'$ of the group of diagonal matrices in $\GL_n(\overline{\CK})$.
We have an isomorphism $\CTbar_w(\overline{\FF}_q) \cong \Hom(X/(\Fr_q - 1)X, \overline{\FF}_q^{\times})$, where $\Fr_q$ is the $q$-power Frobenius.
If we denote by $\mu^{(p)}$ the prime-to-$p$ roots of unity in $\overline{\CK}^{\times}$, then we have an isomorphism:
$$\Hom(\overline{T}_w,\mu^{(p)}) \cong X/(\Fr_q - 1)X \otimes \Hom(\overline{\FF}_q^{\times}, \mu^{(p)}).$$
Noting that $\Fr_q$ acts on $X$ by $qw$, and
applying the duality isomorphism:
$$X/(qw - 1)X \cong \Hom(X'/(qw-1)X', (\QQ/\ZZ)^{(p)})$$
as well as our isomorphism of $(\QQ/\ZZ)^{(p)}$ with $\overline{\FF}_q^{\times}$ arising from our choice of $s$,
we see that $\Hom(\overline{T}_w,\mu^{(p)})$ is naturally isomorphic to $\Hom(X'/(qw-1)X', \mu^{(p)})$.  An element of the latter is precisely a diagonal
matrix $t$, with entries in $\overline{\CK}$, such that $(t^w)^q = t$.  We let $T_q^{w^{-1}}$ denote the set of such matrices.

This construction associates to every $w$, and every character $\varphi: \overline{T}_w \rightarrow \overline{\CK}^{\times}$, an element of $T_q^{w^{-1}}$.
One easily verifies that it sends equivalent pairs $(\CTbar_w,\varphi)$ and $(\CTbar_{w'},\varphi')$ to conjugate diagonal matrices, and further induces a bijection
between $\overline{\CK}$-points of $\Spec \overline{E}_{q,n}$ and those of $\Spec B_{q,n}$.  We thus obtain an isomorphism
of $\overline{E}_{q,n} \otimes \overline{\CK}$ with $B_{q,n} \otimes \overline{\CK}$.  This isomorphism is $\Gal(\overline{\CK}/\CK)$-equivariant and thus
descends to an isomorphism of $\overline{E}_{q,n}[\frac{1}{\ell}]$ with $B_{q,n}[\frac{1}{\ell}]$.

\begin{remark} \label{rem:Langlands} \rm
The choices made in defining the bijection above means that this bijection is compatible with local Langlands in the following sense: let $\pi$ be an irreducible
depth zero generic representation of $G$ over $\overline{\CK}$, and let $\rho$ be its Langlands parameter.  If $K_1$ denotes the kernel of the map 
${\mathcal G}(\OO_G) \rightarrow \overline{\mathcal G}(\FF_q)$, then $\pi^{K_1}$ is an irreducible generic $\overline{\CK}$-representation of $\overline{G}$,
and hence gives rise to a $\overline{\CK}$-point of $\Spec \overline{E}_{q,n}$.  On the other hand, the conjugacy class of the semisimplification of $\rho(\sigma)$
gives a $\overline{\CK}$-point of $\Spec B_{q,n}$.  The bijection constructed above identifies these two points for every choice of $\pi$ and $\rho$.
\end{remark}

Since $B_{q,n}$ and $\overline{E}_{q,n}$ are $\ell$-torsion free, we may regard them as $W(k)$-lattices in $B_{q,n}[\frac{1}{\ell}] \cong \overline{E}_{q,n}[\frac{1}{\ell}]$.
A priori it is not clear that either lattice is contained in the other.  We will show later that in fact these lattices coincide, but this is quite difficult-
it will emerge from the same inductive argument that proves both the weak and strong conjecture in section~\ref{sec:main}.  For the moment, it will suffice to prove something
much weaker.

Recall that one has a symmetrizing form $\theta: \overline{E}_{q,n} \rightarrow W(k)$; the inclusion $B_{q,n} \rightarrow \overline{E}_{q,n}[\frac{1}{\ell}]$ allows us
to regard $\theta$ as a map from $B_{q,n}$ to $\overline{\CK}$.  The goal of the remainder of this section is to prove:

\begin{theorem} \label{thm:B-symmetrizing}
The map $\theta: B_{q,n} \rightarrow \overline{\CK}$ takes values in $W(k)$.
\end{theorem}

As a corollary, we immediately deduce
\begin{corollary} \label{cor:E-B}
Suppose that the isomorphism $B_{q,n}[\frac{1}{\ell}] \cong \overline{E}_{q,n}[\frac{1}{\ell}]$ identifies $\overline{E}_{q,n}$ with a subring of $B_{q,n}$.  Then this isomorphism
identifies $\overline{E}_{q,n}$ with $B_{q,n}$.
\end{corollary}
\begin{proof}
(c.f. Lemma 3.8 of~\cite{gelfand-graev})
If $\overline{E}_{q,n}$ is contained in $B_{q,n}$, then $\theta(be)$ lies in $W(k)$ for all $b \in B_{q,n}$, $e \in \overline{E}_{q,n}$; thus $B_{q,n}$ is contained in the dual lattice
to $\overline{E}_{q,n}$ with respect to $\theta$.  But since $\theta$ is a symmetrizing form on $\overline{E}_{q,n}$, this dual lattice is $\overline{E}_{q,n}$.  Thus $B_{q,n}$ 
and $\overline{E}_{q,n}$ must coincide inside $\overline{E}_{q,n}[\frac{1}{\ell}]$.
\end{proof}

In order to prove Theorem~\ref{thm:B-symmetrizing} we compute the values of $\theta$ on a $W(k)$-spanning set for $B_{q,n}$.  By definitition we have a surjection:
$$W(k)[X']^{S_n} = W(k)[r_1,\dots,r_n,r_n^{-1}]^{S_n} \rightarrow B_{q,n}$$
with kernel $I_{q,n}$.
For each character $\lambda \in X'$, let $N_{\lambda}$ denote the subgroup of $S_n$ normalizing $\lambda$.  Then the elements
$$r_{\lambda} = \frac{1}{\# N_{\lambda}} \sum_{w \in S_n} \lambda^w$$
form a $W(k)$-basis of $W(k)[X']^{S_n}$, as $\lambda$ runs over the elements of $X'$, so their images in $B_{q,n}$ 
(which we also, slightly abusively, denote by $r_{\lambda}$) span $B_{q,n}$ over $W(k)$.

\begin{lemma} 
For $\lambda \in X'$, let $M_{\lambda}$ denote the number of $w \in S_n$ such that the restriction of $\lambda$ to the subgroup $T_q^w$ of $\Hom(X',\overline{\CK}^{\times})$
is trivial.  Then we have:
$$\theta(r_{\lambda}) = \frac{M_{\lambda}}{\# N_{\lambda}}.$$
\end{lemma}

\begin{proof}
Let $x$ be a $\overline{\CK}$-point of $\Spec B_{q,n}$, and let $e_x$ denote the element of $B_{q,n} \otimes \overline{\CK}$ that takes the value $1$ at $x$ and zero at all other
$\overline{\CK}$-points of $\Spec B_{q,n}$.  Our construction of the isomorphism $\overline{E}_{q,n} \otimes \overline{\CK} \cong B_{q,n} \otimes \overline{\CK}$, together with
Proposition~\ref{prop:E-symmetrizing}, shows that 
$$\theta(e_x) = \frac{1}{n!} \sum_{w \in S_n} \frac{N(w^{-1},x)}{\# \overline{T}_w(\FF_q)} = \frac{1}{n!} \sum_{w \in S_n} \frac{N'(w,x)}{\# T_q^{w^{-1}}},$$
where $N'(w,x)$ denotes the number of elements of $T^w_q$ in the equivalence class corresponding to $x$.  It follows that we have:
$$\theta(r_{\lambda}) = \frac{1}{n!} \sum_x r_{\lambda}(x) \sum_{w \in S_n} \frac{N'(w,x)}{\# T_q^{w^{-1}}}.$$
Since $r_{\lambda}(t)$ depends only on the equivalence class of $t \in T^w_q$, we can rewrite this as:
$$\theta(r_{\lambda}) = \frac{1}{n!} \sum_{w \in S_n} \frac{1}{\# T_q^{w^{-1}}} \sum_{t \in T_q^{w^{-1}}} \frac{1}{\# N_{\lambda}} \sum_{v \in S_n} \lambda^v(t).$$
Changing the order of the summation, we obtain:
$$\theta(r_{\lambda}) = \frac{1}{n! \# N_{\lambda}} \sum_{v \in S_n} \sum_{w \in S_n} \frac{1}{\# T_q^{w^{-1}}} \sum_{t \in T_q^{w^{-1}}} \lambda^v(t),$$
and the innermost sum is equal to $0$ if $\lambda^v$ is nontrivial on $T_q^{w^{-1}}$ and equal to $\# T_q^{w^{-1}}$ otherwise.  Thus the sum over $w$ is equal to $M_{\lambda^v}$
which is equal to $M_{\lambda}$.  We thus have $\theta(r_{\lambda}) = \frac{M_{\lambda}}{\# N_{\lambda}}$ as claimed.
\end{proof}

In light of this result, the proof of Theorem~\ref{thm:B-symmetrizing} is reduced to the following result:
\begin{lemma} \label{lemma:count}
For any $\lambda \in X'$, the order of $N_{\lambda}$ divides $M_{\lambda}$.
\end{lemma}

It is clear that the set of $w$ such that $\lambda$ is trivial on $T_w^q$ is stable under conjugation by elements of $N_{\lambda}$, but of course this
action is not faithful, so the divisibility is not immediate.

We begin by observing that $N_{\lambda}$ is the Weyl group of the Levi subgroup of $\GL_n$ centralizing $\lambda$.  This Levi corresponds to a
partition of the $\{1,2,\dots,n\}$ into subsets, and $N_{\lambda}$ is then the subgroup of $S_n$ that preserves this partition.
In particular if $w$ lies in $N_{\lambda}$, then any cycle occuring in the
cycle decomposition of $w$ also lies in $N_{\lambda}$.

Now let $N_{\lambda,w}$ denote the centralizer
of $w$ in $N_{\lambda}$.  Let $O(w)$ be the partition of the set $\{1,\dots,n\}$ into orbits under the action of $w$; then conjugation by $N_{\lambda,w}$ permutes
the orbits of $w$, yielding a map $N_{\lambda,w} \rightarrow \Aut(O(w))$, where $\Aut(O(w))$ is the group of permutations of $O(w)$.

\begin{definition} We will say that $w$ is $N_{\lambda}$-minimal if the map $N_{\lambda,w} \rightarrow \Aut(O(w))$ is injective.
\end{definition}

Note that the property of being $N_{\lambda}$-minimal is stable under $N_{\lambda}$-conjugacy.
Given an arbitrary $N_{\lambda}$-conjugacy class in $S_n$, we will associate an $N_{\lambda}$-minimal conjugacy class in a natural way.
On the level of specific permutations $w$ this construction will depend not just on $w$ but on a particular choice of cycle representation
for $w$.  Here by a ``cycle representation'' of $w$ we mean an unordered collection of expressions of the form $(x_1 \dots x_r)$, with $x_1,\dots,x_r$ distinct
elements of $\{1,\dots,n\}$, that correspond to a disjoint set of cycles whose product is $w$.  To give a cycle representation of $w$ is equivalent to
specifying, for each orbit $x$ of $w$ on $\{1,\dots,n\}$, a distinguished element $x_1$ of the orbit $x$.

Now fix $w \in S_n$, along with a cycle representation of $w$, and let $K$ be the kernel of the map from $N_{\lambda,w}$ to $\Aut(O(w))$.
Then $K$ acts on each orbit $O(w)$; such an orbit $x$ comes from a cycle $(x_1 \dots x_r)$ in our chosen cycle representation of $w$.
Since $K$ centralizes
$w$, it must ``cyclically permute'' the elements of this orbit; that is, the action of $K$ factors through a map $K \rightarrow \ZZ/r\ZZ$, where
$s \in \ZZ/r\ZZ$ acts by sending each $x_i$ to $x_{i+s}$, and the indices are considered modulo $r$.  Let $m$ be the order of the image 
of the map $K \rightarrow \ZZ/r\ZZ$, and set $s = \frac{r}{m}$.
Let $x^{\min}$ be the permutation given by the product of the $m$ disjoint cycles $(x_1 \dots x_s) (x_{s+1} \dots x_{2s}) \dots (x_{r-s+1} \dots x_r)$.
We then define $w^{\min}$ to be the product, over all cycles $x \in O(w)$, of the permutations $x^{\min}$.  It is clear from the construction that if $w$ is $N_{\lambda}$-minimal
then $w^{\min} = w$.

This construction depends on our choice of cycle representation of $w$; in particular
if we represented the cycle $x = (x_1 \dots x_r)$ as $(x_{t+1} \dots x_{t+r})$ instead then we would obtain
the product of cycles 
$$(x_{t+1} \dots x_{t+s}) (x_{t + s +1} \dots x_{t+2s}) \dots (x_{r + t -s + 1} \dots x_{r+t})$$ 
instead of the product:
$$(x_{1} \dots x_{s}) (x_{s +1} \dots x_{2s}) \dots (x_{r -s + 1} \dots x_{r}).$$ Note that the former is $N_{\lambda}$-conjugate to
the latter, via the permutation that, for each $0 \leq a < \frac{r}{s}$ and each $1 \leq b \leq s$, takes $x_{as + b}$ to $x_{t + as + c}$,
where $c$ is the unique integer between $1$ and $s$ such that $as + b$ is congruent to $as + t + c$ modulo $s$.  However, the two permutations are of course
not equal.  Thus changing the cycle representation of $w$ conjugates $w^{\min}$ by an element of $N_{\lambda}$.  In particular the $N_{\lambda}$-conjugacy
class $[w^{\min}]$ depends only on $w$ and not its cycle representation.

On the other hand, if we fix a $v \in N_{\lambda}$,
and a cycle representation of $w$, then conjugating this cycle representation by $v$ gives a cycle representation of $vwv^{-1}$.  Then if we compute $w^{\min}$
and $(vwv^{-1})^{\min}$ using these cycle representations it is easy to see that $(vwv^{-1})^{\min} = v w^{\min} v^{-1}$.  In particular $[w^{\min}]$ depends only
on the $N_{\lambda}$-conjugacy class of $w$.

\begin{lemma} For any $w \in S_n$, $w^{\min}$ is $N_{\lambda}$-minimal.
\end{lemma}
\begin{proof}
Suppose for a contradiction that $w^{\min}$ is not $N_{\lambda}$-minimal, and
let $K$ be the kernel of the map $N_{\lambda,w^{\min}} \rightarrow \Aut(O(w^{\min}))$.  Choose an element $k$ of $K$ other than the identity.  By definition $k$ preserves
every orbit in $O(w)$ and acts nontrivially on at least one such orbit $x = (x_1 \dots x_r)$; we have an $s$ such that $k x_i = x_{i+s}$ for all $i$.  Let $k'$ denote the
permutation that sends $x_i$ to $x_{i+s}$ for all $i$ and fixes all other elements.  Then $k'$ lies in $N_{\lambda}$, since $k$ does and $k'$ is a product of cycles of $k$.
Moreover it is clear that $k'$ commutes with $w^{\min}$.

Our construction of $w^{\min}$ from $w$ implies that the $w^{\min}$-cycle $x$ is contained in a $w$-cycle $x'$ of the form $(x_1 \dots x_{r'})$ for some multiple $r'$ of $r$,
and that the cycles $(x_{r+1} \dots x_{2r})$, etc. are cycles of $w^{\min}$.  Let $k''$ be the permutation that takes $x_i$ to $x_{i+s}$ for all $1 \leq i \leq r'$; then it is clear
that $k''$ centralizes $w$.  We will show that in fact $k''$ lies in $N_{\lambda}$; this gives a contradiction as then we have an element of $N_{\lambda,w}$ that
acts by a shift of length $s$ on the cycle $x'$, meaning that in passing from $w$ to $w^{\min}$ the cycle $x'$ should decompose into cycles of length dividing $s$, and not cycles of length $r$
as we have supposed.

To show that $k''$ lies in $N_{\lambda}$ it suffices to show that for all $i$, $x_{i+s}$ and $x_i$ lie in the same $N_{\lambda}$-orbit.  For $1 \leq i \leq r-s$ this is clear since $k'$ lies in
$N_{\lambda}$.  On the other hand, since $x'$ decomposes into cycles of length $r$ in the cycle decomposition of $w^{\min}$, there is an element of $N_{\lambda}$ that carries $x_i$ to
$x_{i+r}$ for all $i$.  The claim follows.
\end{proof}

The association $w \mapsto [w^{\min}]$ defines an equivalence relation $\sim$ on $S_n$, such that $w \sim v$ if, and only if, $[w^{\min}] = [v^{\min}]$. It is clear
that each equivalence class for $\sim$ is a union of $N_{\lambda}$-orbits.  We will show that in fact each equivalence
class has cardinality equal to $\# N_{\lambda}$.  We begin by fixing an $N_{\lambda}$-minimal $w$.  Then we have an injection $N_{\lambda,w} \rightarrow \Aut(O(w))$.  We will say two
orbits $x,x'$ in $O(w)$ are $N_{\lambda,w}$-equivalent if there is an element of $N_{\lambda,w}$ that takes $x$ to $x'$.  We then have:

\begin{lemma}
Suppose $w$ is $N_{\lambda}$-minimal, and let
$v$ be a permutation of $O(w)$ such that for all $x \in O(w)$, $vx$ is $N_{\lambda,w}$-equivalent to $x$.  Then there is a unique element $\tilde v$ of $N_{\lambda,w}$ whose image in
$\Aut(O(w))$ is $v$.  In particular, $N_{\lambda,w}$ is a product of symmetric groups.
\end{lemma}
\begin{proof}
Uniqueness is clear from the definition of $N_{\lambda}$-minimality.  For existence, fix an orbit $x \in O(w)$.  Then there is an element $v'_x$ of $N_{\lambda,w}$
that takes $x$ to $vx$.  We can then define $\tilde v$ to be the bijection on $\{1,2,\dots,n\}$ that agrees with $v'_x$ on $x$ for all orbits $x$.  Note
that for all $1 \leq i \leq n$, we have ${\tilde v}(i) = v'_x(i)$ for $x$ the $w$-orbit containing $i$; since $v'_x$ is in $N_{\lambda}$
we have $\lambda_i = \lambda_{v'_x(i)} = \lambda_{{\tilde v}(i)}$, so $\tilde v$ lies in $N_{\lambda}$.
\end{proof}

We now fix a particular $N_{\lambda}$-minimal $w$, and a particular cycle representation of $w$.  Since $w$ is $N_{\lambda}$-minimal we may (and do) choose
this cycle representation so that it is preserved by the action of $N_{\lambda,w}$.  Then given any $v \in N_{\lambda,w}$, define ${\tilde w}(v)$
to be the permutation constructed as follows: for orbit of $v$ on $O(w)$, choose an $x$ representing that orbit.  The orbit $x$ then corresponds to a term
$(x_1 \dots x_r)$ in our chosen cycle representation of $w$.  
Let ${\tilde w}(v)_x$ be the permutation $(x_1 \dots x_r \, vx_1 \dots vx_r \dots v^{d-1}x_1 \dots v^{d-1}x_r)$, where $d$ is the order of the $v$-orbit of $x$.
Let ${\tilde w}(v)$ be the product, over a set of representatives $x$ for the orbits of $v$ on $O(w)$, of ${\tilde w}(v)_x$.  Note that as a permutation, ${\tilde w}(v)$
is independent of our choices of representatives $x$ but does depend on our choice of cycle representation of $w$.  On the other hand, our initial choice of cycle representation of $w$,
together with the choices of representatives $x$, gives rise to a cycle representation of ${\tilde w}(v)$.

\begin{lemma}
Let $u$ be an element of $N_{\lambda}$.  Then $u$ conjugates ${\tilde w}(v)$ to ${\tilde w}(v')$ if, and only if, $u$ normalizes $w$ and conjugates $v$
to $v'$.  Moreover, we have ${\tilde w}(v)^{\min} = w$.
\end{lemma}
\begin{proof}
First assume that $u$ normalizes $w$.  Then $u$ actually fixes our chosen cycle representation of $w$, since $w$ is $N_{\lambda}$-minimal.  It is then easy to see from the
construction that ${\tilde w}(uvu^{-1}) = u{\tilde w}(v)u^{-1}$.  .

Conversely, assume $u$ conjugates ${\tilde w}(v)$ to ${\tilde w}(v')$.  Let $x = (x_1 \dots x_r)$ be a cycle in our chosen representation of $w$, such
that the induced cycle of ${\tilde w}(v)$ is $(x_1 \dots x_r \, vx_1 \dots vx_r \dots v^{d-1}x_1 \dots v^{d-1}x_r)$.  Since $u$ conjugates ${\tilde w}(v)$ to ${\tilde w}(v')$the cycle
$(ux_1 \dots ux_r \, uvx_1 \dots uvx_r \dots uv^{d-1}x_1 \dots uv^{d-1}x_r)$ is a cycle of ${\tilde w}(v')$.  This cycle contains a cycle
$(y_1 \dots y_{r'})$ of our chosen representation of $w$.  Thus, by construction of ${\tilde w}(v')$, there is an $s \in \ZZ/dr\ZZ$ such that the sequence:
$$ux_1, \dots ux_r, uvx_1, \dots uvx_r, \dots uv^{d-1}x_1, \dots uv^{d-1}x_r$$ 
coincides with the cyclic shift by $s$ of the sequence:
$$y_1, \dots y_{r'}, vy_1, \dots vy_{r'}, \dots v^{d'-1}y_1, \dots v^{d'-1}y_{r'},$$
where $dr = d'r'$.

Since $u$ and $v$ both lie in $N_{\lambda}$, it follows that for all $1 \leq i \leq r$, and all integers $j$, $x_i$ lies in the same $N_{\lambda}$-orbit as $y_{i+s+jr'}$, where the indices are
taken modulo $r$.  Let $a = (r,r')$.  Then for all $i$, $x_i$ lies in the same $N_{\lambda}$-orbit as $x_{i+a}$.  Thus the permutation that takes $x_i$ to $x_{i+a}$ for all $i$ and fixes all other elements
lies in $N_{\lambda}$.  This permutation clearly normalizes $w$ and fixes all orbits of $w$, so must be the identity since $w$ is $N_{\lambda}$-minimal.  Thus $a = r$, so $r$ divides $r'$.  Similar reasoning shows that
$r'$ divides $r$, so in fact $r$ equals $r'$.

Now for all $1 \leq i \leq r$, $x_i$ is in the same $N_{\lambda}$-orbit as $y_{i+s}$; there is thus an element of $N_{\lambda,w}$ that carries the cycle $(x_1 \dots x_r)$ of $w$ to the cycle $(y_1 \dots y_r)$.
Since we chose our cycle representation of $w$ to be $N_{\lambda,w}$-stable, there is also an element of $N_{\lambda,w}$ that takes $x_i$ to $y_i$ for all $i$.  There is thus an element of $N_{\lambda,w}$
that takes $x_i$ to $x_{i+s}$ for all $i$, and fixes all other elements of $\{1, \dots, n\}$.  Since $w$ is minimal, this is impossible unless $r$ divides $s$.

We have thus established that $u$ takes the cycle $x = (x_1 \dots x_r)$ of $w$ to the cycle $(v^ey_1, \dots, v^ey_r)$ for some $e$, which is also a cycle of $w$.  Since $x$ was arbitrary, $u$ preserves the cycles of
$w$ and thus normalizes $w$.  But now we have ${\tilde w}(uvu^{-1}) = u{\tilde w}(v) u^{-1} = {\tilde w}(v')$, and it is easy to see that this implies that $uvu^{-1} = v'$.

For the final claim, let $x= (x_1 \dots x_r)$ be a cycle in our chosen representation of $w$, contained in the cycle $(x_1 \dots x_r vx_1 \dots vx_r \dots v^{d-1}x_1 \dots v^{d-1} x_r)$ of ${\tilde w}(v)$.
The subgroup of $N_{\lambda,w}$ preserving the latter cycle acts on it by cyclic shifts, and minimality of $w$ implies that $r$ divides the length of any of these shifts.  On the other hand it is clear that
the permutation that agrees with $v$ on the set $\{x_1,\dots,x_r,vx_1,\dots,vx_r,\dots,v^{d-1}x_1, \dots v^{d-1}x_r\}$ and is the identity elsewhere induces a shift of length $r$ on this cycle.  Our construction
of ${\tilde w}(v)^{\min}$ thus demands that we break this cycle of ${\tilde w}(v)$ into cycles of length $r$.  Doing this for all cycles of ${\tilde w}(v)$ recovers $w$.
\end{proof}

We now show:
\begin{lemma} Suppose $w$ is $N_{\lambda}$-minimal and $w' \sim w$.  Then there exists $v \in N_{\lambda}(w)$ such that $w'$ is $N_{\lambda}$-conjugate to ${\tilde w}(v)$.
\end{lemma}
\begin{proof}
We first construct a cycle representation of $w'$ such that the induced cycle representation of $(w')^{\min}$ is $N_{\lambda,(w')^{\min}}$-invariant.  To do this, first fix any orbit of $w'$
and choose a representation of the corresponding cycle; we then obtain representations of one or more cycles in $(w')^{\min}$, all of which are $N_{\lambda}$-conjugate.  We then proceed inductively: for each
orbit $x$ of $w'$, choose a cycle representation arbitrarily and consider the resulting cycles of $(w')^{\min}$.  If these cycles are not $N_{\lambda}$-conjugate to other cycles of $(w')^{\min}$ that have already been
constructed, there is nothing further to do and we may proceed to the next orbit of $w'$.  If they are conjugate to cycles we have already constructed, it need not be the case that the corresponding cycle 
{\em representations} are $N_{\lambda}$-conjugate to those already extant (they may differ by a cyclic shift).  However, adjusting our choice of cycle representation
of $x$ by a suitable shift we may arrange that this holds.  Proceeding inductively we arrive at a $(w')^{\min}$ and an $N_{\lambda,(w')^{\min}}$-invariant cycle representation of it.

Now for each cycle $x$ of $w'$, our chosen decompositions give $x = (x_1 \dots x_{rs})$ in $w'$, for some integers $r,s$ such that the corresponding cycles of $(w')^{\min}$ are
$(x_1 \dots x_r)$, $(x_{r+1} \dots x_{2r})$, etc.  Let $v'_x$ be the permutation that takes $x_i$ to $x_{i+r}$ for all $i$ (indices modulo $rs$); then $v'_x$ lies in $N_{\lambda}$.  Taking $v'$ to be the product
over the orbits $x$ of the $v'_x$ we obtain an element of $N_{\lambda,(w')^{\min}}$ such that $w' = \widetilde{(w')^{\min}}(v')$.  Now if $w' \sim w$ then there exists a $u \in N_{\lambda}$ such that
$u(w')^{\min} u^{-1} = w$; taking $v= uv'u^{-1}$ we find that $uw'u^{-1} = {\tilde w}(v)$.
\end{proof}

\begin{corollary}
Suppose $w$ is $N_{\lambda}$-minimal.  The number of $w'$ such that $w' \sim w$ is equal to the order of $N_{\lambda}$.
\end{corollary}
\begin{proof}
The previous lemmas show that the set of such $w'$ is the union of the $N_{\lambda}$-conjugacy classes of ${\tilde w}(v)$, as $v$ runs over a set of representatives for the conjugacy classes
in $N_{\lambda,w}$.  For each such $v$ the size of its $N_{\lambda}$-conjugacy class is equal to $\frac{\# N_{\lambda}}{\# N_{\lambda,v}}$.  For each $v$, the index of $N_{\lambda,w}$ in $N_{\lambda,v}$
is equal to the size of the $N_{\lambda,w}$-conjugacy class $C_v$ of $v$.  Thus the total number of such $w'$ is the sum:
$${\# N_{\lambda}} \sum_v \frac{\# C_v}{\# N_{\lambda, w}}$$
which is clearly equal to $\# N_{\lambda}$.
\end{proof}

We now relate the equivalence $\sim$ to $M_{\lambda}$.  Specifically, we observe:
\begin{proposition}
Suppose that $w \sim w'$.  Then $\lambda$ is trivial on $T^w_q$ if, and only if, $\lambda$ is trivial on $T^{w'}_q$.
\end{proposition}
\begin{proof}
It suffices to show this in the case where $w' = w^{\min}$ (for some chosen cycle representation of $w$), as we can deduce any other case from this one and $N_{\lambda}$-conjugacy.

Let $S_{\lambda}$ be the set of $N_{\lambda}$-orbits on $\{1,\dots,n\}$, and $f: \{1,\dots,n\} \rightarrow S_{\lambda}$ the map that sends an element to its $N_{\lambda}$-orbit.  There exists
a map $g: S_{\lambda} \rightarrow \ZZ$ such that on the diagonal matrix $t$ with entries $t_1,\dots, t_n$, we have $\lambda(t) = \prod_i t_i^{g(f(i))}$.

An element of $T^w_q$ is a diagonal matrix whose entries $t_i$ satisfy $t_{w(i)} = t_i^q$ for all $i$.  In particular, for each $i$, $t_i$ is a $q^{d_i} - 1$st root of unity, where $d_i$ is the
size of the $w$-orbit of $i$.  In particular, $\lambda$ is trivial on $T^w_q$ if, and only if, for all $i$ the sum:
$$\Sigma_i = \sum\limits_{j=0}^{d_i - 1} q^j g(f(w^j(i)))$$
is divisible by $q^{d_i} - 1$.

In $w^{\min}$ the $w$-orbit of $i$ breaks up as a union of $N_{\lambda}$-conjugate orbits, each of size $r$.  In particular for each $j$, the elements $w^j(i)$ and $w^{j+r}(i)$ lie in
the same $N_{\lambda}$-orbit, so $g(w^j(i)) = g(w^{j+r}(i))$.  This means that the sum $\Sigma_i$ can be rewritten as:
$$\Sigma_i = (1 + q^r + \dots + q^{d_i-r}) \sum\limits_{j=0}^{r-1} q^j g(f(w^j(i))).$$
In particular $\Sigma_i$ is divisible by $q^{d_i} - 1$ if, and only if, the sum:
$$\sum\limits_{j=0}^{r-1} q^j g(f(w^j(i)))$$
is divisible by $q^r - 1$.  But this is precisely the condition for $\lambda$ to be trivial on $w^{\min}$.
\end{proof}

From this it follows that the quotient $\frac{M_{\lambda}}{\# N_{\lambda}}$ counts the number of $N_{\lambda}$-minimal orbits of $w$ in $S_n$ such that $\lambda$ is trivial on $T_q^w$.  In particular
this quotient is an integer.  This completes the proof of Lemma~\ref{lemma:count} and hence of Theorem~\ref{thm:B-symmetrizing}.

\section{Deformation theory} \label{sec:deformations}

In this section we examine the local deformation theory of a representation $\rhobar: G_F \rightarrow \GL_n(k)$.
As in previous sections, let $I_F^{(\ell)}$ denote the prime to $\ell$ part of the inertia group of $F$, and
fix a topological generator $\tsigma$ of $I_F/I_F^{(\ell)}$ and a Frobenius element $\tFr$ in $W_F/I_F^{(\ell)}$.

We first recall some results of Clozel-Harris-Taylor:

\begin{proposition}[\cite{CHT}, Lemmas 2.4.11-2.4.13] \label{prop:CHT}
Let $\taubar$ be an irreducible representation of $I_F^{(\ell)}$ over $k$, and let $G_{\taubar}$ be the
subgroup of $G_F$ that preserves $\taubar$ under conjugation.  Then:
\begin{enumerate}
\item $\taubar$ lifts uniquely to a representation $\tau$ of $I_F^{(\ell)}$ over $W(k)$.
\item $\tau$ extends uniquely to a representation of $I_F \cap G_{\taubar}$ of determinant prime to $\ell$.
\item $\tau$ extends (non-uniquely) to a representation of $G_{\taubar}$.
\end{enumerate}
If we fix a representation $\tau$ of $G_{\taubar}$ as in part (3), we obtain an action
of $G_{\taubar}/I_F^{(\ell)}$ on $\Hom_{I_F^{(\ell)}}(\tau,\rho)$ for any $G_F$-module $\rho$.
Moreover, we have a direct sum decomposition of $G_F$-modules:
$$\rho \cong \bigoplus_{[\taubar]} \Ind_{G_{\taubar}}^{G_F} [\Hom_{I_F^{(\ell)}}(\tau,\rho) \otimes \tau],$$
where $\taubar$ runs over $G_F$-conjugacy classes of irreducible representations of $I_F^{(\ell)}$
over $k$.
\end{proposition}

Fix, for each $G_F$-conjugacy class of $\taubar$, a $\tau$ as in the proposition.  Suppose we are given a
representation $\rho_A: G_F \rightarrow \GL_n(A)$.  We then obtain a direct sum decomposition:
$$\rho_A = \bigoplus_{[\taubar]} \Ind_{G_{\taubar}}^{G_F} [\Hom_{I_F^{(\ell)}}(\tau,\rho_A) \otimes \tau].$$
It is clear that $\Hom_{I_F^{(\ell)}}(\tau,\rho_A)$ is a free $A$-module for all $\tau$, and that the collection
of $G_{\taubar}$-representations $\Hom_{I_F^{(\ell)}}(\tau,\rho)_A$ determines the representation $\rho_A$ up to isomorphism.

\begin{definition} A {\em pseudo-framing} of a continuous representation $\rho_A: G_F \rightarrow \GL_n(A)$ is a choice,
for each $\taubar$, of basis for each $\Hom_{I_F^{(\ell)}}(\tau,\rho_A)$.  A {\em pseudo-framed deformation}
of a continuous representation $\rhobar: G_F \rightarrow \GL_n(k)$ (together with a chosen pseudo-framing) is a lift
$\rho_A: G_F \rightarrow \GL_n(A)$ of $\rhobar$, together with a pseudo-framing of $\rho_A$ that lifts the chosen
pseudo-framing of $\rho$.
\end{definition}

Fix a $\rhobar$ and a pseudo-framing of $\rhobar$, and, for each $\taubar$, let $\rhobar_{\taubar}$ be the
$G_{\taubar}$-representation $\Hom_{I_F^{(\ell)}}(\tau,\rhobar)$.  Let $R_{\rhobar}^{\diamond}$ be the completed tensor product
$$\hat{\bigotimes}_{[\taubar]} R_{\rhobar_{\taubar}}^{\Box}$$
of the universal framed deformation rings of the $\rhobar_{\taubar}$.  Over each such ring we have the
universal framed deformation $\rho_{\taubar}^{\Box}$ of $\rhobar_{\taubar}$.

Using these, we construct a
representation:
$$\rho^{\diamond} := \bigoplus_{[\taubar]} \Ind_{G_{\taubar}}^{G_F} [\rho_{\taubar}^{\Box} \otimes \tau]$$
that has a natural pseudo-framing induced by the universal framings of the representations $\rho^{\Box}_{\taubar}$.
One easily verifies that the pair $R_{\rho}^{\diamond}, \rho^{\diamond}$ is a universal object for
pseudo-framed deformations of $\rho$.

For each $\taubar$, the formal group ${\mathcal G}_{\rhobar_{\taubar}}^{\Box}$ acts on $\Spf R_{\rho_{\taubar}}^{\Box}$
by ``change of frame''.  Let ${\mathcal G}^{\diamond}_{\rhobar}$ be the product of the ${\mathcal G}_{\rhobar_{\taubar}}^{\Box}$.
Then ${\mathcal G}^{\diamond}_{\rhobar}$ acts on $\Spf R_{\rhobar}^{\diamond}$ by ``change of pseudo-framing''.

For computational purposes it is often easier to work with $R_{\rhobar}^{\diamond}$ rather than $R_{\rhobar}^{\Box}$,
as $R_{\rhobar}^{\diamond}$ can be made quite explicit.  The two rings are related in a natural way: 
one has a ring $R_{\rhobar}^{\Box,\diamond}$ that is universal for triples consisting of
a deformation $\rho$ of $\rhobar$, a framing of $\rho$ lifting that of $\rhobar$, and a pseudo-framing of $\rho$
lifting that of $\rhobar$.  Then $\Spf R_{\rhobar}^{\Box,\diamond}$ is a (split) ${\mathcal G}_{\rhobar}^{\diamond}$-torsor
over $\Spf R_{\rhobar}^{\Box}$ and a (split) ${\mathcal G}_{\rhobar}^{\Box}$-torsor over $\Spf R_{\rhobar}^{\diamond}$.

We immediately deduce:
\begin{cor}
The ring $R_{\rhobar}^{\Box}$ is a reduced, $\ell$-torsion free local complete intersection.
\end{cor}
\begin{proof}
The construction above shows that it suffices to prove the same claim with $R_{\rhobar}^{\Box}$
replaced by $R_{\rhobar}^{\diamond}$.  But the latter is a completed tensor product of rings of
the form $R_{\rhobar_{\taubar}}^{\Box}$, and each of these is isomorphic to the completion of a ring
of the form $R_{q,n}$ (with $q$ and $n$ depending on $\taubar$)
at a maximal ideal.  The result thus follows from the results of section~\ref{sec:tame}.
\end{proof}

Moreover, we may canonically identify both the ${\mathcal G}_{\rhobar}^{\Box}$-invariant elements of $R_{\rhobar}^{\Box}$ 
and the ${\mathcal G}_{\rhobar}^{\diamond}$-invariant elements of $R_{\rhobar}^{\diamond}$ with
the ${\mathcal G}_{\rhobar}^{\Box} \times {\mathcal G}_{\rhobar}^{\diamond}$-invariant elements of
$R_{\rhobar}^{\Box,\diamond}$.  In particular these spaces of invariants are naturally isomorphic.

Given a choice
of framing of $\rho^{\diamond}$, we get a map $R_{\rhobar}^{\Box} \rightarrow R_{\rhobar}^{\diamond}$.  When restricted to
${\mathcal G}_{\rhobar}^{\Box}$-invariants this map is the isomorphism of 
$\left(R_{\rhobar}^{\Box}\right)^{{\mathcal G}_{\rhobar}^{\Box}}$ with
$\left(R_{\rhobar}^{\diamond}\right)^{{\mathcal G}_{\rhobar}^{\diamond}}$ constructed above.  Summarizing, we have:

\begin{lemma} \label{lem:compare invariants}
For any choice of framing of $\rho^{\diamond}$, the induced map: $R_{\rhobar}^{\Box} \rightarrow R_{\rhobar}^{\diamond}$
identifies the ${\mathcal G}_{\rhobar}^{\Box}$-invariant elements of $R_{\rhobar}^{\Box}$ with the
${\mathcal G}_{\rhobar}^{\diamond}$-invariant elements of $R_{\rhobar}^{\diamond}$.  (In particular the image of this
set of invariant elements is saturated in $R_{\rhobar}^{\diamond}$.)
\end{lemma}

\section{The rings $R_{\nu}$} \label{sec:global galois}

Let $\rhobar: W_F/I_F^{(\ell)} \rightarrow \GL_n(k)$ be a representation.  Then we have a corresponding map
$x: R_{q,n} \rightarrow k$, with kernel ${\mathfrak m}$.  It follows easily from the universal property of
the pair $(R_{q,n}, \rho_{F,n})$ that the completion $(R_{q,n})_{\mathfrak m}$ is isomorphic to
$R_{\rhobar}^{\diamond}$, and that this isomorphism is induced by the base change of $\rho_{F,n}$ to
$(R_{q,n})_{\mathfrak m}$.  In other words, $R_{q,n}$ is a global object that interpolates the formal deformation
rings $R_{\rhobar}^{\diamond}$ for $\rhobar$ trivial on $I_F^{(\ell)}$.

We would like to construct similar objects for $\rhobar$ whose restriction to $I_F^{(\ell)}$ is nontrivial.
Let us define:

\begin{defn} An $\ell$-inertial type is a representation $\nu$ of $I_F^{(\ell)}$ over $k$ that
extends to a representation of $W_F$.
\end{defn}

Note that (as $I_F^{(\ell)}$ is a profinite group of pro-order prime to $\ell$), such a representation
lifts uniquely to a representation of $I_F^{(\ell)}$ over $W(k)$, and this lift also extends to a 
representation of $W_F$.  We will thus consider an $\ell$-inertial type $\nu$ as a representation over 
$W(k)$ rather than over $k$ whenever it is convenient to do so. 

Now fix an $\ell$-inertial type $\nu$, and for each irreducible representation
$\taubar$ of $I_F^{(\ell)}$ over $k$, let $n_{\taubar}$ be the multiplicity of $\taubar$ in $\nu$
(note that $n_{\taubar}$ depends only on the $W_F$-conjugacy class of $\taubar$.)  Let $W_{\taubar}$
be the subgroup of $W_F$ that fixes $\taubar$ under conjugation, let $F_{\taubar}$ be the fixed field
of $W_{\taubar}$, and let $q_{\taubar}$ denote the cardinality of the residue field of $F_{\taubar}$.

We define $R_{\nu}$ to be the tensor product:
$$R_{\nu} := \bigotimes_{\taubar} R_{q_{\taubar},n_{\taubar}}$$
where $\taubar$ runs over a set of representatives for the $W_F$-conjugacy classes of irreducible
representations appearing in $\nu$.  For each $\taubar$ we have a representation $\rho_{F_{\taubar},n_{\taubar}}$
over $R_{q_{\taubar},n_{\taubar}}$, which we regard as a representation over $R_{\nu}$ in the obvious way.

Define the representation $\rho_{\nu}: W_F \rightarrow \GL_n(R_{\nu})$ as follows:
$$\rho_{\nu} := \bigoplus_{\taubar} \Ind_{W_{\taubar}}^{W_F} \rho_{F_{\taubar},n_{\taubar}} \otimes \tau,$$
where $\taubar$ runs over a set of representative for the $W_F$-conjugacy classes of irreducible representations
appearing in $\nu$, and for each such $\taubar$, we have chosen an extension $\tau$ of $\taubar$
to a representation $W_F \rightarrow \GL_n(W(k))$ as in Proposition~\ref{prop:CHT}.  Note that
$\rho_{\nu}$ inherits a pseudo-framing from the natural framings of the $\rho_{F_{\taubar},n_{\taubar}}$,
and that the restriction of $\rho_{\nu}$ to $I_F^{(\ell)}$ is given by $\nu$.

For a map $x: R_{\nu} \rightarrow k$, the specialization $(\rho_{\nu})_x$ is a pseudo-framed
representation $W_F \rightarrow \GL_n(k)$, whose restriction to $I_F^{(\ell)}$ is given by $\nu$.
This defines a bijection between $k$-points of $\Spec R_{\nu}$ and such pseudo-framed representations.
Moreover, it follows directly from the constructions of $R_{\nu}$ and $R_{(\rho_{\nu})_x}^{\diamond}$
that the completion of $R_{\nu}$ at the maximal ideal corresponding to $x$ is naturally isomorphic
to $R_{(\rho_{\nu})_x}^{\diamond}$, in a manner compatible with the universal family on the latter.

Moreover, the universal property for each $R_{q_{\taubar},n_{\taubar}}$ immediately yields:
\begin{prop} \label{prop:universal}
For any finitely generated, $\ell$-adically separated $W(k)$-algebra $A$, and any pseudo-framed, $\ell$-adically continuous
representation $\rho: W_F \rightarrow \GL_n(A)$ whose restriction to $I_F^{(\ell)}$ is isomorphic to $\nu$,
there is a unique map: $R_{\nu} \rightarrow A$ such that $\rho$ is the base change of $\rho_{\nu}$.
\end{prop}

For each $\taubar$, the group $\GL_{n_{\taubar}}$ acts on $R_{q_{\taubar},n_{\taubar}}$.  Let ${\mathcal G}_{\nu}$
be the product of the $\GL_{n_{\taubar}}$; then ${\mathcal G}_{\nu}$ acts on $\Spec R_{\nu}$ by ``changing the
pseudo-frame''.

\section{Maps from $Z_{[L,\pi]}$ to $R_{\nu}$}

Now fix a pair $(L,\pi)$, where $L$ is a Levi subgroup of $\GL_n(F)$ and
$\pi$ is an irreducible supercuspidal $k$-representation of $L$.  The mod $\ell$ semisimple local
Langlands correspondence of Vigneras~\cite{vigss} attaches to $\pi$ a semisimple $k$-representation
$\rho$ of $W_F$.  Let $\overline{\nu}$ be the restriction of $\rho$ to $I_F^{(\ell)}$.  Then $\overline{\nu}$
lifts uniquely to a $W(k)$-representation $\nu$ of $I_F^{(\ell)}$, and we have:

\begin{prop}
The irreducible $\overline{\CK}$-representations of $\GL_n(F)$ that are objects of $\Rep_{W(k)}(\GL_n(F))_{[L,\pi]}$ correspond,
via local Langlands, to the $\overline{\CK}$-representations of $W_F$ whose restriction to $I_F^{(\ell)}$
is isomorphic to $\nu$.
\end{prop}
\begin{proof}
This is an easy consequence of the compatibility of Vigneras' mod $\ell$ correspondence with reduction mod $\ell$.
\end{proof}

This proposition shows that for any $\overline{\CK}$-point $x$ of $\Spec R_{\nu}$, the representation $\rho_x$
corresponds, via local Langlands (and Frobenius semsimplification if necessary) to an irreducible $\overline{\CK}$-representation $\Pi_x$
in $\Rep_{W(k)}(\GL_n(F))_{[L,\pi]}$, and hence to a $\overline{\CK}$-point of $\Spec Z_{[L,\pi]}$.  It is a natural question
to ask whether this map is induced by a map $Z_{[L,\pi]} \rightarrow R_{\nu}$.  Indeed, we conjecture:

\begin{conjecture}[Weak local Langlands in families] \label{conj:weak}
There is a map $Z_{[L,\pi]} \rightarrow R_{\nu}$ such that the induced map on $\overline{\CK}$-points
takes a point $x$ of $\Spec R_{\nu}$ to the $\overline{\CK}$-point of $Z_{[L,\pi]}$ that gives the action
of $Z_{[L,\pi]}$ on the representation $\Pi_x$ corresponding to $\rho_x$ by local Langlands.  (We will say such a
map is {\em compatibile with local Langlands}.)
\end{conjecture}

Since $R_{\nu}$ is reduced and $\ell$-torsion free, such a map is unique if it exists.  Note also that the image
of any element of $Z_{[L,\pi]}$ under such a map is invariant under the action of ${\mathcal G}_{\nu}$, and so any such map
must factor through the subalgebra $R_{\nu}^{\inv}$ of ${\mathcal G}_{\nu}$-invariant elements of $R_{\nu}$.  We further conjecture:

\begin{conjecture}[Strong local Langlands in families] \label{conj:strong}
There is an isomorphism $Z_{[L,\pi]} \cong R_{\nu}^{\inv}$ such that the composition
$$Z_{[L,\pi]} \rightarrow R_{\nu}^{\inv} \rightarrow R_{\nu}$$
is compatible with local Langlands.
\end{conjecture}

If one completes at a maximal ideal of $R_{\nu}$, corresponding to a representation $\rhobar$ of $W_F$ over $k$,
and uses Lemma~\ref{lem:compare invariants} to relate the
invariant elements of $R_{\rhobar}^{\Box}$ and $R_{\rhobar}^{\diamond}$, one recovers Conjectures 7.5 and 7.6
of~\cite{bernstein3}.  In particular (c.f. Theorem 7.9 of~\cite{bernstein3}), Conjecture~\ref{conj:weak} above implies
the ``local Langlands in families'' conjecture of Emerton-Helm (conjecture 1.1.3 of~\cite{emerton-helm}).

These conjectures should be viewed as relating ``congruences'' between admissible representations (which are in some
sense encoded in the structure of $Z_{[L,\pi]}$) with ``congruences'' between representations of $W_F$ (encoded in
$R_{\nu}$).  Since inverting $\ell$ destroys information about such congruences, one expects such conjectures to be relatively
straightforward with $\ell$ inverted.  We will show that this is indeed the case.

First, note that any map:
$$Z_{[L,\pi]} \otimes \overline{\CK} \rightarrow R_{\nu} \otimes \overline{\CK}$$
that is compatible with local Langlands is Galois equivariant, and hence descends to a map
$$Z_{[L,\pi]}[\frac{1}{\ell}] \rightarrow R_{\nu}[\frac{1}{\ell}]$$
compatible with local Langlands.  It thus suffices to show:

\begin{theorem} \label{thm:invert ell}
There is a map $Z_{[L,\pi]} \otimes \overline{\CK} \rightarrow R_{\nu} \otimes \overline{\CK}$
compatible with local Langlands (and therefore a corresponding map over $\CK$.)
Moreover, the image of this map is $R_{\nu}^{\inv} \otimes \overline{\CK}$.
\end{theorem}

To prove this, we first work on the level of connected components.  We have an isomorphism:
$$Z_{[L,\pi]} \otimes \overline{\CK} \cong \prod\limits_{M,\tpi} \tZ_{(M,\tpi)},$$
by Theorem~\ref{thm:bernstein char 0}, where $(M,\tpi)$ varies over the inertial equivalence classes of pairs
that reduce modulo $\ell$ to $(L,\pi)$.  Thus the connected components of $\Spec Z_{[L,\pi]} \otimes \overline{\CK}$
are in bijection with the pairs $(M,\tpi)$.  Via local Langlands, these correspond to representations of $I_F$.
More precisely, let $\Pi$ be an admissible representation of $G$, let $\rho: W_F \rightarrow \GL_n(\overline{\CK})$
correspond to $\Pi$ via local Langlands, and let $\trho: W_F \rightarrow \GL_n(\overline{\CK})$ be the
representation of $W_F$ corresponding to $\tpi$ via local Langlands.  Then $\Pi$
belongs to the block corresponding to $(M,\tpi)$ if and only if the restriction of $\rho^{\sss}$ to $I_F$ coincides
with the restriction of $\trho$ to $I_F$.

On the other hand, it is an easy consequence of Proposition~\ref{prop:tame constant inertia} that as $x$ varies over
$\overline{\CK}$-points of $\Spec R_{\nu}$, the restriction of $\rho_{\nu,x}^{\sss}$ to $I_F$ is constant on
connected components of $\Spec R_{\nu} \otimes \overline{\CK}$.  We can thus let $R_{\nu}^{\trho}$ be the direct factor
of $R_{\nu} \otimes \overline{\CK}$ corresponding to the union of the connected components of $\Spec R_{\nu} \otimes \overline{\CK}$
on which the restriction of $\rho_{\nu,x}^{\sss}$ to $I_F$ is isomorphic to the restriction of $\trho$ to $I_F$.  We will see
later that $\Spec R_{\nu}^{\trho}$ is in fact connected.

It then suffices to construct, for each $(M,\tpi)$, an isomorphism:
$$\tZ_{(M,\tpi)} \rightarrow (R_{\nu}^{\trho})^{\inv}$$
compatible with local Langlands.  Since $(M,\tpi)$ is only well-defined up to inertial equivalence, we may assume that
$\tpi$ has the form:
$$\tpi \cong \bigotimes_i \tpi_i^{\otimes r_i},$$
where the $\tpi_i$ are pairwise inertially inequivalent representations of $\GL_{n_i}(F)$.  Unwinding the Bernstein-Deligne
description of $\tZ_{(M,\tpi)}$, we obtain an isomorphism:
$$\tZ_{(M,\tpi)} \cong \bigotimes_i \overline{\CK}[X_{i,1}^{\pm 1}, \dots, X_{i,r_i}^{\pm 1}]^{S_{r_i}},$$
where the symmetric group $S_{r_i}$ acts by permuting the elements $X_{i,1}, \dots, X_{i,r_i}$.

For each $i$, and any $\alpha \in \overline{\CK}$, let $\chi_{i,\alpha}$ denote the unramified character of $\GL_{n_i}(F)$
that takes the value $\alpha$ on any element of $\GL_{n_i}(F)$ with determinant $\unif_F$.
An irreducible $\Pi$ in $\Rep_{\overline{\CK}}(M,\tpi)$ has supercuspidal support $(M,\tpi')$ for some $\tpi'$ of the form:
$$\tpi' \cong \bigotimes_i \bigotimes\limits_{j=1}^{r_i} \tpi_i \otimes \chi_{i,\alpha_{i,j}}$$
for suitable $\alpha_{i,j}$.  Then the $d$th elementary symmetric function in $X_{i,1}, \dots, X_{i,r_i}$, considered as an element
of $\tZ_{(M,\tpi)}$, acts on $\Pi$ via the $d$th elementary symmetric function in the $\alpha_{i,1}^{f'_i}, \dots, \alpha_{i,r_i}^{f'_i}$,
where $f'_i$ is the order of the group of unramified characters $\chi$ such that $\tpi_i \otimes \chi$ is isomorphic to $\tpi_i$.

For each $i$, the irreducible representation $\trho_i$ of $W_F$ corresponding to $\tpi_i$ via local Langlands decomposes, when
restricted to $I_F$, as a direct sum of distinct irreducible representations of $I_F$, all of which are $W_F$-conjugate.  Fix an irreducible
representation $\ttau_i$ of $I_F$ contained in $\trho_i$, and let $W_i$ be the normalizer of $\ttau_i$ in $W_F$.  Then there is a
unique way of extending $\ttau_i$ to a representation of $W_i$ such that the induction of the resulting extension to $W_F$ is isomorphic
to $\trho_i$.  (Note that this implies that $W_i$ has index $f'_i$ in $W_F$.)

This choice of extension of $\ttau_i$ to $W_i$ gives rise to an action of $W_i$ on the space $\Hom_{I_F}(\ttau_i, \rho_{\nu})$.
The quotient of this space that lives over $R_{\nu}^{\trho}$ is a free $R_{\nu}^{\trho}$-module of rank $r_i$, with an unramified
action of $W_i$.

Let $\tFr_i$ be a Frobenius element of $W_i$, and let $P_i(x) = \sum\limits_{j = 0}^{r_i} a_{i,j} X^j$ be
the characteristic polyomial of $\tFr_i$ on $\Hom_{I_F}(\ttau_i,\rho_{\nu})$ (over $R_{\nu}^{\trho}$).
Consider the map $\tZ_{(M,\tpi)} \rightarrow R_{\nu}^{\trho}$ that sends the
$d$th elementary symmetric function in $X_{i,1}, \dots, X_{i,r_i}$ to the element $(-1)^d a_{i,r_i - d}$
of $R_{\nu}^{\trho}$.  One verifies easily that this map is compatible with local Langlands.

It remains to show that $(R_{\nu}^{\trho})^{\inv}$ is generated by the images of these elements.  Given a polynomial
$P_i$ of degree $r_i$, with coefficients in a ring $R$, we can associate to it the unramified $R$-representation
$M_i(P_i)$ of $W_i$ on which $\tFr_i$ acts via the companion matrix of $P_i$.  The representation $\rho(\{P_i\})$
given by:
$$\rho(\{P_i\}) = \bigoplus_i \Ind_{W_i}^{W_F} M_i(P_i) \otimes \ttau_i$$
is then an $R$-point of $\Spec R_{\nu}^{\trho}$.  In this way we obtain a natural map:
$$R_{\nu}^{\trho} \rightarrow \bigotimes_i \overline{\CK}[Y_{i,1}, \dots, Y_{i,r_i}]$$
that in particular takes the element $(-1)^d a_{i,r_i-d}$ of $R_{\nu}^{\trho}$ to $Y_{i,d}$.
On the other hand, it is easy to see that for every $y$ in $(\Spec R_{\nu}^{\trho})(\overline{\CK})$, there is a point
$x$ in $(\Spec R_{\nu}^{\trho})(\overline{\CK})$ arising from a collection of polynomials $\{P_i(x)\}$ such that $y$
is in the closure of the $G_{\nu}$-orbit of $x$.  It follows that the map:
$$R_{\nu}^{\trho} \rightarrow \bigotimes_i \overline{\CK}[Y_{i,1}, \dots, Y_{i,r_i}]$$
is injective on $(R_{\nu}^{\trho})^{\inv}$.  Therefore $(R_{\nu})^{\trho})^{\inv}$ is generated by the elements $a_{i,r_{i-d}}$,
completing the proof.
 
It is not hard to go slightly further, and show:
\begin{theorem} \label{thm:normalize}
The image of $Z_{[L,\pi]}$ in $R_{\nu}[\frac{1}{\ell}]$ under the map of Theorem~\ref{thm:invert ell} lies in the normalization
of $R_{\nu}$.
\end{theorem}
\begin{proof}
Fix an element $x$ of $Z_{[L,\pi]}$, and let $y$ be its image in $R_{\nu}[\frac{1}{\ell}]$.
Let $A$ be a discrete valuation ring that is a $W(k)$-algebra, with field of fractions $K$ of characteristic zero, and fix
a map $R_{\nu} \rightarrow A$.  This corresponds to a pseudo-framed representation $\rho_A$ of $W_F$.  Let $\Pi_K$ denote
the admissible $K$-representation corresponding to $\rho_A \otimes_A K$ via local Langlands.  Since $\rho_A \otimes_A K$ admits
an $A$-lattice, so does $\Pi_K$.  In particular the action of $x$ on $\Pi_K$ is via an element of $A$, so $y$ maps to
an element of $A$ under the map $R_{\nu}[\frac{1}{\ell}] \rightarrow K$.  Since this is true for every $A$ and every map $R_{\nu} \rightarrow A$,
$y$ lives in the normalization of $R_{\nu}$ as claimed.
\end{proof}

\section{Main results} \label{sec:main}
The main objective of this section (and, indeed, the paper) is to show the following:

\begin{theorem} \label{thm:main}
Suppose that Conjecture~\ref{conj:weak} holds for all $\GL_m(F)$, $m \leq n$, and Conjecture~\ref{conj:strong} holds for $m < n$.  Then:
\begin{enumerate}
\item The map $\overline{E}_{q,n}[\frac{1}{\ell}] \rightarrow B_{q,n}[\frac{1}{\ell}]$ of section~\ref{sec:E-B} induces an isomorphism of $\overline{E}_{q,n}$
with $B_{q,n}$, and
\item Conjecture~\ref{conj:strong} holds for $\GL_n(F)$.
\end{enumerate}
\end{theorem}

We begin by proving the first claim, using the weak conjecture for $\GL_n$ in depth zero.  Let $Z_n^0$ be the product of
the depth zero blocks of $\Rep_{W(K)}(G)$.  The weak conjecture then gives rise to a map $Z_n^0 \rightarrow S_{q,n}$
compatible with the local Langlands correspondence.  The subalgebra of $Z_n^0$ consisting of elements that are constant on inertial equivalence classes
is isomorphic to $\overline{E}_{q,n}$, by Proposition~\ref{prop:finite bernstein image}.  By compatibility with local Langlands together with
Proposition~\ref{prop:tame constant inertia} and Proposition~\ref{prop:inertial saturation} the image of $\overline{E}_{q,n}$ in $S_{q,n}$ is
contained in $B_{q,n}$, and the induced map: $\overline{E}_{q,n}[\frac{1}{\ell}] \rightarrow B_{q,n}[\frac{1}{\ell}]$ is the map considered in section~\ref{sec:E-B}.
It thus follows from Corollary~\ref{cor:E-B} that the map $\overline{E}_{q,n} \rightarrow B_{q,n}$ is an isomorphism.

We now turn to the second claim.  Fix a mod $\ell$ supercuspidal inertial equivalence class $[L,\pi]$, corresponding to an $\ell$-inertial type $\nu$,
and note that we have tensor factorizations:
$$Z_{[L,\pi]} \cong \bigotimes_i Z_{[L_i,\pi_i]}$$
$$R_{\nu} \cong \bigotimes_{\taubar} R_{q_{\taubar},n_{\taubar}}$$
where the $[L_i,\pi_i]$ are simple blocks.  The former factorization is compatible with parabolic induction and the latter
arises from the direct sum decomposition:
$$\rho_{\nu} = \bigoplus_{\taubar} \Ind_{W_{\taubar}}^{W_F} \rho_{F_{\taubar},n_{\taubar}} \otimes \tau.$$
Since simple blocks correspond to types $\nu$ with only one $n_{\taubar}$ nonzero, these factorizations are compatible,
in the sense that if we have maps $Z_{[L_i,\pi_i]} \rightarrow R_{\nu_i}$ for each $i$ that are compatible with local Langlands,
then their tensor product gives a map $Z_{[L,\pi]} \rightarrow R_{\nu}$ compatible with local Langlands.  Thus both
Conjecture~\ref{conj:weak} and Conjecture~\ref{conj:strong} reduce to the corresponding conjectures on simple blocks.  We thus henceforth
assume that $[L,\pi]$ is of the form $[L_n,\pi_n]$ with $\pi_n \cong \pi_1^{\otimes n}$ for a supercuspidal representation $\pi_1$. 
Following section~\ref{sec:bernstein} we set $Z_n = Z_{[L_n,\pi_n]}$.  The corresponding $R_{\nu_n}$ is then isomorphic to
$R_{q_{\taubar},n}$ for some fixed $\taubar$.

We first consider the case in which $n$ is not $q_{\taubar}$-relevant.  Let $\nu$ be the maximal $q_{\taubar}$-relevant partition of $n$.
We have a commutative diagram:
$$
\begin{array}{ccc}
Z_n & \rightarrow & R^{\inv}_{q_{\taubar},n}\\
\downarrow & & \downarrow\\
\otimes_i Z_{\nu_i} & \rightarrow & \otimes_i R^{\inv}_{q_{\taubar},\nu_i}
\end{array}
$$
in which the horizontal maps are those arising from the weak conjecture, the left-hand vertical map is $\Ind_{\nu}$, and the right-hand vertical map is
induced by the map $\Spec \otimes_i R_{q_{\taubar},\nu_i} \rightarrow R_{q_{\taubar},n}$ that takes a collection $(\Fr_i,\sigma_i)$ of matrices with
$\Fr_i \sigma_i \Fr_i^{-1} = \sigma_i^{q_{\taubar}}$ to the pair $(\oplus_i \Fr_i, \oplus_i \sigma_i)$.

The horizontal maps are isomorphisms after inverting $\ell$, and our hypotheses imply that the lower horizontal map is an isomorphism integrally.  Moreover
the left-hand vertical map is injective with saturated image by Theorem~\ref{thm:bernstein ind} and the discussion in the paragraph following it.  It follows
immediately that the top horizontal map must also be an isomorphism.

We now assume that $n$ is $q_{\taubar}$-relevant (that is, it lies in $\{1,e_{q_{\taubar}}, \ell e_{q_{\taubar}}, \dots \})$ let $m$ be the largest element of this
set that is strictly less than $n$.  Set $j = \frac{n}{m}$.

We have a subalgebra $\overline{E}_{q^{f'},n,1}$ of $Z_m$ and compatibility with local Langlands shows that $q^{f'} = q_{\taubar}$.  Thus the map
$Z_n \rightarrow R_{q_{\taubar},n}$ induces a map $\overline{E}_{q_{\taubar},n,1} \rightarrow R_{q_{\taubar},n}$.  Reasoning as in the depth zero setting
we see that the image of this map is contained in $B_{q_{\taubar},n,1}$.  It seems likely that the resulting map $\overline{E}_{q_{\taubar},n,1} \rightarrow B_{q_{\taubar},n,1}$
is the one considered in section~\ref{sec:E-B}, but we do not prove this here.  Instead we use the fact that we have shown these two rings to be abstractly isomorphic,
together with the following lemma:

\begin{lemma}
Let $\overline{E}$ be a finite rank, reduced, $\ell$-torsion free $W(k)$-algebra, and let $f: \overline{E} \rightarrow \overline{E}$ be an injection.
Then $f$ is an isomorphism.
\end{lemma}
\begin{proof}
Clearly $f$ is an isomorphism after inverting $\ell$.  On the other hand, the hypotheses guarantee that $\overline{E}[\frac{1}{\ell}]$ is a product
of finite extensions of $\CK$, and $f$ is a $\CK$-linear automorphism of this product.  In particular there is some power of $f$ that is the identity.
\end{proof}

We thus conclude that the map $Z_n \rightarrow R_{q_{\taubar},n}$ coming from the weak conjecture induces an isomorphism
of $\overline{E}_{q_{\taubar},n,1}$ with $B_{q_{\taubar},n,1}$.

Now consider the commutative diagram:
$$
\begin{array}{ccc}
K & \rightarrow & K'\\
\downarrow & & \downarrow\\
Z_n & \rightarrow & R_{q_{\taubar},n}^{\inv}\\
\downarrow & & \downarrow\\
Z_m^{\otimes j} & \rightarrow & (R_{q_{\taubar},m}^{\inv})^{\otimes j}\\
\end{array}
$$
in which the horizontal maps are induced by the weak conjecture, the lower left vertical map is $\Ind_{m,n}$, the lower right vertical map
is the one taking a collection of pairs $(\Fr_i,\sigma_i)$ to their direct sum, and $K$ and $K'$ are the kernels of the lower left and lower right
vertical maps, respectively.  As in the previous case, all horizontal maps become isomorphisms after inverting $\ell$ and the bottom horizontal map
is an isomorphism integrally.

By Proposition~\ref{prop:kernel} $K$ is contained in the subalgebra $\overline{E}_{q_{\taubar},n,1}[\Theta_{n,n}^{\pm 1}]$ of $Z_n$, and the image
of this subalgebra in $R_{q_{\taubar},n}$ is saturated.  It follows that the map from $K$ to $K'$ is an isomorphism: if $x$ is an element of $K'$,
then for some $a$, the product $\ell^a x$ is in the image of $K$.  But then $\ell^a x$ is in 
the image of $\overline{E}_{q_{\taubar},n,1}[\Theta_{n,n}^{\pm 1}]$, so $x$ is as well.  On the other hand, the image of $K$ in
$\overline{E}_{q_{\taubar},n,1}$ is saturated (as $K$ is the kernel of a map of rings that have no $\ell$-torsion), so $x$ must lie in the image
of $K$.

Let $r$ be an element of $R_{q_{\taubar},n}^{\inv}$, and let $r'$ be its image in $(R_{q_{\taubar},m}^{\inv})^{\otimes j}$.  There is then
an element $y$ of $Z_m^{\otimes j}$ whose image under the bottom horizontal map is $r'$.  Since the map $Z_n \rightarrow R_{q_{\taubar},n}^{\inv}$ is
an isomorphism after inverting $\ell$, there exists $a$ such that $\ell^a y$ is in the image of $\Ind_{m,n}$.  

By Theorem~\ref{thm:near saturation}, there exist $\tilde y$ in $Z_n$ and $x$ in $\overline{E}_{q_{\taubar},n,1}[\Theta_{n,n}^{\pm 1}]$ such that
$\Ind_{m,n}(x) = \ell^b(\Ind_{m,n}(\tilde y) - y)$.  Let $s$ be the image of $\tilde y$ in $R_{q_{\taubar},n}^{\inv}$.  The image
of $\ell^b(s-r)$ in $(R_{q_{\taubar},m}^{\inv})^{\otimes j}$ coincides with the image of $\Ind_{m,n}(x)$.  Thus $\ell^b(s-r)$ lies in
the image of $\overline{E}_{q_{\taubar},n,1}[\Theta_{n,n}^{\pm 1}]$.  Since this image is saturated, the element $s-r$ also lives in this image.
Thus the map $Z_n \rightarrow R_{q_{\taubar},n}^{\inv}$ is surjective, so it is an isomorphism.

We have thus completed the proof of Theorem~\ref{thm:main}.  In~\cite{converse} we show that the strong conjecture for $\GL_{n-1}$
implies the weak conjecture for $\GL_n$.  Together with Theorem~\ref{thm:main} and the fact that the strong conjecutre for $\GL_1$
is an easy consequence of local class field theory, we obtain an unconditional proof both of the strong conjecture, and of the existence of an
isomorphism $\overline{E}_{q,n} \cong B_{q,n}$.  We refer the reader to the final section of~\cite{converse} for the details.

\begin{remark} \rm The isomorphism of $\overline{E}_{q,n}$ with $B_{q,n}$ is an interesting result in finite group theory in its own right.  We are aware of
no proof other than the one presented here; it is an interesting question to find a purely group-theoretic proof of this result.
\end{remark}

\section{Affine Curtis Homomorphisms}

Having established both Conjectures~\ref{conj:weak} and~\ref{conj:strong} we now turn to an interesting consequence of Conjecture~\ref{conj:weak}.
Fix a $w$ in $S_n$ (which we identify with the Weyl group of $\mathcal G$).  The conjugacy class of $w$ gives rise to a conjugacy class of
nonsplit, unramified tori in $\mathcal G$; we let $\CT_w$ denote a representative of this conjugacy class.  In particular we have
$\CT_w \cong \prod_{w_i} \Res_{F_i/F} {\mathbb G}_m$, where the product is over the cycles $w_i$ of $w$ and $F_i/F$ is unramified of degree equal to the length of $w_i$.
Let $d$ be the order of $w$ in $S_n$.

Let $X$ be the character group of $\CT_w$, and let $\CT_w^L$ denote the algebraic group $\Hom(X',{\mathbb G}_m) \rtimes \ZZ/d\ZZ$ (regarded as an algebraic group over $W(k)$),
where the action of $1 \in \ZZ/d\ZZ$ on $X'$ is via $w^{-1}$.  Then $\CT_w^L$ is the $L$-group of $\CT_w$.  Moreover, if we identify $\GL_n$ (over $W(k)$) with the $L$-group
of ${\mathcal G}$ in such a way that $X'$ becomes identified with the character group of the diagonal torus in $\GL_n$, then we have a natural $L$-homomorphism from
$\CT_w^L$ to $\GL_n$ that takes $\Hom(X',{\mathbb G}_m)$ to the diagonal torus and takes $1 \in \ZZ/n\ZZ$ to $w^{-1}$.  This allows us to transfer a Langlands parameter
$\rho_w: W_F \rightarrow \CT_w^L(\overline{\CK})$ for $\CT_w$ to a Langlands parameter $\rho: W_F \rightarrow \GL_n(\overline{\CK})$ for ${\mathcal G}$.

It will be useful to understand the interaction between this transfer and the block decompositions for $\Rep_{W(k)}(T_w)$ and $\Rep_{W(k)}(G)$.  Note that 
$$T_w = \CT_w(F) = \Hom(X,(F^{\ur})^{\times})^{\tFr},$$
where $\tFr$ is a fixed Frobenius element of $W_F$, and its action on $X$ is via $w$.  Let $T_w^{(\ell)}$ denote the subgroup $\Hom(X,(\OO_{F^{\ur}}^{\times})^{(\ell)})^{\tFr}$ of $T_w$,
where $(\OO_{F^{\ur}}^{\times})^{(\ell)}$ denotes the elements of pro-order prime to $\ell$ in $\OO_{F^{\ur}}^{\times}$.  Then $T_w^{(\ell)}$ is profinite, of pro-order prime to $\ell$,
and the quotient $T_w/T_w^{(\ell)}$ is a discrete group.  Indeed, explicitly, one has:
$$T_w/T_w^{(\ell)} \cong \prod_{w_i} F_i^{\times}/(\OO_{F_i}^{\times})^{(\ell)} \cong \prod_{w_i} (\ZZ\cdot \unif \times \FF_{q^i}^{\times}),$$
where $\unif$ is a uniformizer of $F$ (hence also of $F_i$.)

The blocks of $\Rep_{W(k)}(T_w)$ are thus given by characters $\chi^{(\ell)}: T_w^{(\ell)} \rightarrow W(k)^{\times}$.  Choose an extension $\chi$ of $\chi^{(\ell)}$ to a character 
$T_w \rightarrow W(k)^{\times}$.  Then ``twisting by $\chi$'' induces an equivalence of categories between the block of $\Rep_{W(k)}(T_w)$ corresponding to the trivial character of $T_w^{(\ell)}$
and the block corresponding to $\chi^{(\ell)}$.  Denote the centers of these blocks by $Z_{w,1}$ and $Z_{w,\chi^{(\ell)}}$, respectively; our choice of $\chi$ then gives an isomorphism of
$Z_{w,1}$ with $Z_{w,\chi^{(\ell)}}$.

On the other side of the Langlands correspondence, the local Langlands correspondence for tori associates to $\chi$ a Langlands parameter $\tilde \nu_w: W_F \rightarrow \CT_w^L(\overline{\CK})$;
the restriction $\nu_w$ of $\tilde \nu_w$ to $I_F^{(\ell)}$ depends only on $\chi^{(\ell)}$.  Consider the functor that associates to a $W(k)$-algebra $R$ the set of parameters
$W_F \rightarrow \CT_w^L(R)$ whose restriction to $I_F^{(\ell)}$ is equal to $\nu_w$.  This functor is easily seen to be representable by a finite type affine scheme $\Spec R_{\nu}^w$, and
there is a universal Langlands parameter $\rho_{w,\nu}: W_F \rightarrow \CT_w^L(R_{\nu}^w)$.  Note that the torus $\Hom(X',\GG_m) \subseteq \CT_w^L$ acts on
$\Spec R_{\nu}^w$ by conjugation; let $(R_{\nu}^w)^{\inv}$ be the subring of $R_{\nu}^w$ invariant under this action.

We then have the following Proposition, which can be seen as an analogue of Conjecture~\ref{conj:strong}
for the nonsplit torus $T_w$:

\begin{proposition} \label{prop:LLIFFT}
There is a unique isomorphism $$\LL_w: Z_{w,\chi^{(\ell)}} \rightarrow (R_{\nu}^w)^{\inv}$$ 
which is compatible with the local Langlands correspondence for tori, in the sense that
for any Langlands parameter $\rho: W_F \rightarrow \CT_w^L(\overline{\CK})$, corresponding to a character $\chi_{\rho}$ of $T_w$, and any $z \in Z_{w,\chi^{(\ell)}}$,
the value of $\chi_{\rho}$ at $z$ is equal to the value of $\LL_w$ at the point of $\Spec (R_{\nu}^w)$ corresponding to $\rho$.
\end{proposition}
\begin{proof}
Any parameter $W_F \rightarrow \CT_w^L(R)$ of type $\nu_w$ differs from $\tilde \nu_w$ by a parameter $W_F \rightarrow \CT_w^L(R)$ that is trivial on $I_F^{(\ell)}$.  
Thus ``twisting by $\nu_w$'' induces an isomorphism of $\Spec R_{\nu}^w$ with $\Spec R_1^w$, where $1$ is the trivial character of $I_F^{(\ell)}$.  On $\overline{\CK}$-points,
this isomorphism is compatible with the local Langlands correspondence for tori and the ``twisting by $\chi$'' isomorphism of $Z_{w,1}$ with $Z_{w,\chi^{(\ell)}}$.  We can thus reduce
to the case where $\chi^{(\ell)}$ and $\nu_w$ are the trivial character.

In this case we can be very explicit: on the one hand, we have isomorphisms:
$$Z_{w,1} = W(k)[T_w/T_w^{(\ell)}] \cong W(k)[\Hom(X,(F^{\ur})^{\times}/(\OO_{F^{\ur}}^{\times})^{(\ell)})^{\tFr}],$$
where $\tFr$ acts on $X$ via $w$ and on $\OO_{F^{\ur}}^{\times}$ in the usual way.  Let $\unif$ be a uniformizer of $F$ corresponding to our Frobenius element $\tFr$.  We then have a
canonical isomorphism:
$$(F^{\ur})^{\times}/(\OO_{F^{\ur}}^{\times})^{(\ell)} \cong \ZZ\cdot \unif \times \overline{\FF}_q^{\times},$$
where $\tFr$ acts trivially on the first factor and by $q$th powers on the second.  We thus obtain an isomorphism:
$$Z_{w,1} \cong W(k)[\Hom(X,\ZZ)^w] \otimes W(k)[\Hom(X/(qw-1)X,\overline{\FF}_q^{\times})].$$

On the other side of the Langlands correspondence, fix a generator $\tsigma$ of $I_F/P_F$.  Then a Langlands parameter $W_F \rightarrow \CT_w^L$ trivial on $I_F^{(\ell)}$
is determined by the images of $\tFr$ and $\tsigma$; these form a pair of diagonal matrices $\CF$ and $\sigma$ such that $\CF w^{-1} \sigma (\CF w^{-1})^{-1} = \sigma^q$.
Since $\CF$ and $\sigma$ commute, this condition is equivalent to the condition $\sigma^{w^{-1}} = \sigma^q$.  Thus $\Spec R_1^w$ decomposes as a product:
$$\Spec R_1^w \cong \Spec W(k)[X'] \times \Spec W(k)[X'/(q-w)X'],$$
where the first factor parameterizes $\CF$ and the second parameterizes $\sigma$.  The conjugation action of $t \in \Hom(X',{\mathbb G}_m)$ on this product fixes the second factor
and acts by multiplication by $t^{w^{-1} - 1}$ on the first.  We thus obtain a product decomposition:
$$\Spec (R_1^w)^{\inv} \cong \Spec W(k)[X'/(1 - w)X'] \times \Spec W(k)[X'/(q - w)X'].$$

On the first factor, the isomorphism of $Z_{w,1}$ with $(R_1^w)^{\inv}$ is induced by the isomorphism $\Hom(X,\ZZ)^w \cong X'/(w-1)X'$.  On the second factor we have to work a bit harder.
Note that $qw - 1$ divides $q^r - 1$, where $r$ is a multiple of the order of $w$.  Thus $\Hom(X/(qw-1)X, \overline{\FF}_q^{\times})$ is isomorphic to
$\Hom(X/(qw-1)X, \FF_{q^r}^{\times})$.  Our choice of $s$ gives rise to a system of generators for $\FF_{q^r}^{\times}$ for all $r$, compatible with respect to norm maps;
we can thus identify $\Hom(X/(qw-1)X, \FF_{q^r}^{\times})$ with the kernel of $qw^{-1} - 1$ on $X'/(q^r - 1)X'$, via the isomorphism 
$$X'/(q^r - 1)X' \cong \Hom(X/(q^r - 1)X, \ZZ/(q^r - 1)\ZZ).$$
Finally, multiplication by $1 + qw^{-1} + \dots + q^{r-1}w^{1-r}$ identifies this kernel with $X'/(qw^{-1} - 1)X'$.  The resulting isomorphism of $\Hom(X/(qw-1)X,\overline{\FF}_q^{\times})$
with $X'/(qw^{-1} - 1)X'$ is independent of $r$, and gives the desired map from the second factor of $Z_{w,1}$ to the second factor of $(R_1^w)^{\inv}$.  One checks easily
that the resulting isomorphism is compatible with local Langlands.
\end{proof}

The $L$-homomorphism of $\CT_w^L$ into $\GL_n$ takes Langlands parameters for $T_w$ to Langlands parameters for $G$.  If the former has type $\nu_w$, then so does the latter (where
we regard $\nu_w$ as an $\ell$-inertial type by embedding it in $\GL_n(W(k))$ by identifying $\Hom(X',{\mathbb G}_m)$ with the diagonal matrices.)  Thus this $L$-homomorphism induces
a map $R^{\nu} \rightarrow R_{\nu}^w$ that takes $R^{\inv}_{\nu}$ to $(R_{\nu}^w)^{\inv}$.  Combining this with Proposition~\ref{prop:LLIFFT} and Conjecture~\ref{conj:weak}, we obtain
a map:
$$eZ_n \rightarrow Z_{w,\chi^{(\ell)}},$$
where $e$ is the idempotent of $Z_n$ corresponding to the $\ell$-inertial type $\nu$.  On $\overline{\CK}$-points this map takes a point of $\Spec Z_{w,\chi^{(\ell)}}$ 
corresponding to a character with Langlands parameter $\rho$ to the point of $\Spec eZ_n$ corresponding to the Langlands parameter obtained by composing $\rho$ with
the $L$-homomorphism of $\CT_w^L$ into $\GL_n$.

On the other hand, if we fix a generic character $\Psi$ of the unipotent radical $U$ of $G$, and let $\Gamma$ be the module $\cInd_U^G \Psi$, then it follows from results in~\cite{bernstein3}
that the natural map $eZ_n \rightarrow \End_{W(k)[G]}(\Gamma)$ is an isomorphism.  We can thus view the map $eZ_n \rightarrow Z_{w,\chi^{(\ell)}}$ as the affine group analogue
of a Curtis homomorphism.  Since the Curtis homomorphisms have such a nice interpretation via Deligne-Lusztig theory, it is natural to ask if a similar phenomenon is at play here:

\begin{question}
Does there exist an adjoint pair of functors:
$$i_w: {\mathcal D}^b(\Rep_{W(k)}(T_w)) \rightarrow {\mathcal D}^b(\Rep_{W(k)}(G))$$
$$r_w: {\mathcal D}^b(\Rep_{W(k)}(G)) \rightarrow {\mathcal D}^b(\Rep_{W(k)}(F))$$
such that $r_w(\Gamma)$ is a shift of the induction $\cInd_e^{T_w} 1$, and the induced homomorphism:
$$Z_n \rightarrow Z_w$$
is the product over suitable idempotents of the ``affine Curtis homomorphisms'' constructed above?
Moreover, is there a natural geometric construction of such an adjoint pair?
\end{question}


\begin{thebibliography}{}
\bibitem[BD]{BD}
J. Bernstein and P. Deligne, \emph{Le ``centre'' de Bernstein}, in
\emph{Representations des groups redutifs sur un corps local, Traveaux en cours},
(P. Deligne ed.), Hermann, Paris, 1--32.

\bibitem[BK]{gelfand-graev}
C. Bonnaf{\'e} and R. Kessar, \emph{On the endomorphism algebras of modular Gelfand-Graev representations},
J. Algebra 320 (2008), no. 7, 2847-2870.

\bibitem[BR]{BR}
C. Bonnaf{\'e} and R. Rouquier, \emph{Cat{\'e}gories d{\'e}riv{\'e}es et vari{\'e}t{\'e}s de Deligne-Lusztig,}
Pub. Math. IHES 97 (2003), 1--59.

\bibitem[Ch]{choithesis}
S.-H, Choi, \emph{Local deformation 
lifting spaces of mod $\ell$ Galois representations,}
Ph.D. Thesis, Harvard University, 2009.

\bibitem[CHT]{CHT}
L. Clozel, M. Harris, and R. Taylor, \emph{Automorphy for some $\ell$-adic lifts of automorphic
mod $\ell$ representations}, Pub. Math. IHES 108 (2008), 1--181.

\bibitem[Du]{dudas}
O. Dudas, \emph{Deligne-Lusztig restriction of a Gelfand-Graev module},
Ann. Sci. {\'E}c. Norm. Sup{\'e}r. 42 (2009), no. 4, 653--674.

\bibitem[EH]{emerton-helm}
M. Emerton and D. Helm, \emph{The local Langlands correspondence for $\GL_n$ in families},
Ann. Sci. {\'E}c. Norm. Sup{\'e}r. 47 (2014), no. 4, 655--722.

\bibitem[H1]{bernstein1}
D. Helm, \emph{The Bernstein center of the category of smooth $W(k)[\GL_n(F)]$-modules},
Forum Math. Sigma (2016), e11, 98pp.

\bibitem[H2]{bernstein3}
D. Helm, \emph{Whittaker models and the integral Bernstein center for $\GL_n(F)$},
Duke. Math. Journal 165 (2016), no. 9, 1597--1628.

\bibitem[HM]{converse}
D. Helm and G. Moss, \emph{Converse theorems and the local Langlands correspondence in families},
Invent. Math. 214 (2018), no. 2, 999--1022.

\bibitem[Sh]{shotton}
J. Shotton, \emph{The Breuil-M{\'e}zard conjecture when $\ell \neq p$},
Duke Math. J. 167 (2018), no. 4, 603--678.

\bibitem[Vi]{vigss}
M.-F. Vigneras, \emph{Correspondance de Langlands semi-simple pour $\GL(n,F)$ modulo $\ell \neq p$},
Invent. Math. 144 (2001), no. 1, 177--223.



\end{thebibliography}
\end{document}